\documentclass[twoside]{amsart}

\numberwithin{equation}{section}
\usepackage{amsmath,amssymb,amsfonts,amsthm,latexsym}
\usepackage{graphicx,color,enumerate}
\usepackage[margin=1.2in]{geometry}
\usepackage[all]{xy}

\newtheorem{theorem}{Theorem}[section]

\newtheorem{proposition}[theorem]{Proposition}
\newtheorem{corollary}[theorem]{Corollary}

\theoremstyle{definition}
\newtheorem{definition}[theorem]{Definition}
\newtheorem{example}[theorem]{Example}
\theoremstyle{remark}
\newtheorem{remark}[theorem]{Remark}

\numberwithin{equation}{section}

\newcommand{\C}{\mathbb{C}}
\newcommand{\Sp}{\mathbb{S}}

\newcommand{\R}{\mathbb{R}}

\newcommand{\Z}{\mathbb{Z}}

\newcommand{\SL}{\operatorname{SL}}

\newcommand{\lcm}{\operatorname{lcm}}

\newcommand{\git}{/\!\!/}
\newcommand{\bs}{\boldsymbol}

\newcommand{\Hilb}{\operatorname{Hilb}}

\DeclareMathOperator*{\Res}{Res}
\newcommand{\Vand}{\operatorname{V}}
\begin{document}

\title{The Hilbert series and $a$-invariant of circle invariants}

\author[L.~E.~Cowie]{L.~Emily Cowie}
\address{Department of Mathematics,
303 Lockett Hall,
Louisiana State University,
Baton Rouge, LA 70803}
\email{lcowie1@lsu.edu}

\author[H.-C.~Herbig]{Hans-Christian Herbig}
\address{Departamento de Matem\'{a}tica Aplicada,
Av. Athos da Silveira Ramos 149, Centro de Tecnologia - Bloco C, CEP: 21941-909 - Rio de Janeiro, Brazil}
\email{herbig@qgm.au.dk}

\author[D.~Herden]{Daniel Herden}
\address{Department of Mathematics, Baylor University,
One Bear Place \#97328,
Waco, TX 76798-7328, USA}
\email{Daniel\_Herden@baylor.edu}

\author[C.~Seaton]{Christopher Seaton}
\address{Department of Mathematics and Computer Science,
Rhodes College, 2000 N. Parkway, Memphis, TN 38112}
\email{seatonc@rhodes.edu}

\thanks{C.S. was supported by the E.C.~Ellett Professorship in Mathematics.}

\begin{abstract}
Let $V$ be a finite-dimensional representation of the complex circle $\C^\times$ determined by a weight vector
$\bs{a}\in\Z^n$. We study the Hilbert series $\Hilb_{\bs{a}}(t)$ of the graded algebra $\C[V]^{\C_{\bs{a}}^\times}$ of
polynomial $\C^\times$-invariants in terms of the weight vector $\bs{a}$ of the $\C^\times$-action. In particular, we give
explicit formulas for $\Hilb_{\bs{a}}(t)$ as well as the first four coefficients of the Laurent expansion of
$\Hilb_{\bs{a}}(t)$ at $t=1$. The naive formulas for these coefficients have removable singularities when weights pairwise
coincide. Identifying these cancelations, the Laurent coefficients are expressed using partial Schur polynomial that
are independently symmetric in two sets of variables. We similarly give an explicit formula for the $a$-invariant of $\C[V]^{\C_{\bs{a}}^\times}$ in the case that this algebra is Gorenstein. As an application, we give methods to identify
weight vectors with Gorenstein and non-Gorenstein invariant algebras.
\end{abstract}

\maketitle
\tableofcontents


\section{Introduction}
\label{sec:Intro}

Let $V$ be a finite-dimensional representation of a complex reductive group $G$ and let
$R=\C[V]^G$ denote the algebra of $G$-invariant polynomials. It is well known that this algebra
is a finitely generated graded algebra $R=\oplus_{m=0}^\infty R_m$ such that $R_0 = \C$.
The \emph{Hilbert series} of $R$ is the generating function
\[
    \Hilb_R(t)   =   \sum\limits_{m=0}^\infty \dim_\C(R_m)\: t^m.
\]
The Hilbert series is known to be rational with a pole at $t = 1$ of order
$\dim(R)$, the Krull dimension of $R$. It therefore admits a Laurent expansion of the form
\begin{equation}
\label{eq:DefGammas}
    \Hilb_R(t)  =   \sum\limits_{m=0}^\infty \gamma_m(R)(1 - t)^{m - \dim R},
\end{equation}
see \cite[Proposition 1.4.5 and Lemma 1.4.6]{DerskenKemperBook}.
The \emph{$a$-invariant} $a(R)$ of $R$ is defined to be the degree of $\Hilb_R(t)$, i.e. the
degree of the numerator minus the degree of the denominator.

The Hilbert series $\Hilb_R(t)$ contains important imformation about the algebra $R$
and is a relatively accessible quantity. It is used in constructive invariant theory, e.g., when computing generators and relations for $R$ (see for example \cite[Section 2.6]{DerskenKemperBook} and \cite[Chapter 2]{Popov}).
Similarly, the coefficients $\gamma_m(R)$ are often of significance. For instance
when $G$ is a finite group, $\gamma_0(R) = 1/|G|$
and $\gamma_1(R)$ determines the number of pseudoreflections in $G$; see
\cite[Lemma 2.4.4]{SturmfelsBook} or \cite[Sections 2.4 and 2.6]{BensonBook}.
When the coefficient field is $\mathbb{F}_p$, an analogous formula for $\gamma_1(\mathbb{F}_p[V]^G)$
in terms of the stabilizers of hyperplanes in $V$ was conjectured by Carlisle and Kropholler and
proven by Benson and Crawley--Boevey
\cite{BensonCrawley-Boevey}; see \cite[Sections 2.6 and 3.13]{BensonBook}.

When $R = \C[V]^G$ is the invariant ring of a reductive group $G$, less is known about the
meaning of the $\gamma_m(R)$. When $G = \SL_2(\C)$ and $V$ is irreducible, Hilbert gave an
explicit formula for $\gamma_0(R)$. Popov has shown \cite[3.3 Theorem 5]{Popov} that if $G$
is connected and semisimple, then for all but finitely many representations $V$,
$2\gamma_1(\C[V]^G)/\gamma_0(\C[V]^G) = \dim G$. See also \cite[Remark 4.6.16]{DerskenKemperBook}
and the references given there.

In this paper, we investigate the Hilbert series and first four Laurent coefficients when $G$ is a circle $\C^\times$. The techniques and results are parallel to
those in \cite{HerbigSeatonHilbSympCirc} where Hilbert series of symplectic circle quotients have been studied. The main observation in these calculations is that the naive formulas for the Laurent coefficients have removable singularities when certain weights pairwise coincide. By construction, the expressions for the $\gamma_m$'s are invariant with respect to permutations of the positive weights and permutations of the negative weights.
After removing the singularities in the expressions, the Laurent coefficients are written
in terms of a kind of partial Schur polynomial of the weights, which are independently symmetric in the positive and negative
weights. These calculations lead us in particular to a neat formula for the $a$-invariant of Gorenstein circle invariants; see
Corollary~\ref{cor:a-invar}. The phenomenon of removable singularities
in the formulas for Laurent coefficients appears to be a general feature of calculations of Laurent coefficients in the case
of reductive $G$ and deserves further attention.

Part of the motivation for studying the Hilbert series of $\C[V]^{\C^\times}$ is related to the question of
which representations of $\C^\times$ have Gorenstein invariant rings.
Because the canonical module of a ring of torus invariants is determined
explicitly in terms of covariants \cite[Theorem 6.4.2]{BrunsHerzog}, it is known that $\C[V]^{\C^\times}$
is Gorenstein for some but not all representations.  If the representation is unimodular, then the Gorenstein
property holds \cite[Corollary 6.4.3]{BrunsHerzog}, but unlike the case of finite groups that contain no
pseudoreflections \cite{WatanabeGor1,WatanabeGor2} this condition is not necessary. This approach to the
Gorenstein question, however, requires the computation of invariants and covariants, which can be computationally
expensive. Because $\C[V]^{\C^\times}$ is a normal domain by \cite[Proposition 6.4.1]{BrunsHerzog}, and
the Hochster-Roberts Theorem \cite{HochsterTori,HochsterRoberts} implies that $\C[V]^{\C^\times}$ is Cohen-Macaulay,
the Gorenstein property can be determined from $\Hilb_{\C[V]^{\C^\times}}(t)$ alone using the characterization of
Stanley \cite{Stanley} recalled in Equation~\eqref{eq:StanleyGoren} below. Our understanding of
$\Hilb_{\C[V]^{\C^\times}}(t)$ gives a method of checking the Gorenstein condition without computing the invariants.
In some cases, we can resolve this question without computing $\Hilb_{\C[V]^{\C^\times}}(t)$, see
Corollaries~\ref{cor:GorensteinInteger}, \ref{cor:GorenN2}, and \ref{cor:GorenK1}.

The contents of this paper are as follows.
After reviewing the necessary background and establishing notation in Section~\ref{sec:Back},
we compute the Hilbert series in Section~\ref{sec:HilbSer}; see Theorem~\ref{thrm:HilbSer}. As in
\cite{HerbigSeatonHilbSympCirc}, particular attention is paid to the case of a \emph{degenerate}
representation in which negative weights appear with multiplicity. This formula suggests an algorithm for
computing the Hilbert series as a rational function in the generic case, which we describe in
Section~\ref{sec:DirtyMethod}. We introduce a generalization of Schur polynomials in Section~\ref{sec:Schur}
in order to give explicit descriptions of the first few $\gamma_m$ in Section~\ref{sec:Laurent}.
In Section~\ref{sec:Goren}, we discuss the Gorenstein property of $\C[V]^G$ in this case and
give an explicit formula for the $a$-invariant when this property holds. In
Appendix~\ref{app:CMAlg}, we demonstrate how the Laurent coefficients of a Cohen-Macaulay domain
are related to the degrees of the elements of a Hironaka decomposition.


\section*{Acknowledgements}

Herbig and Seaton would like to thank Baylor University, and Herden and Seaton would like to thank
the Instituto de Matem\'{a}tica Pura e Aplicada (IMPA) for hospitality during work contained here.
This paper developed from Cowie's senior seminar project in
the Rhodes College Department of Mathematics and Computer Science, and the
authors gratefully acknowledge the support of the department and college for these
activities. We benefited from lecture notes of Leonid Petrov and follow his terminology
for the Laurent-Schur polynomials.


\section{Background and notation}
\label{sec:Back}

Let $V$ be a finite-dimensional unitary $\C^\times$-module. Choosing a basis for $V$ with respect to
which the $\C^\times$-action is diagonal, we can describe the action with a \emph{weight vector}
$\bs{a} = (a_1,a_2,\ldots,a_n)\in\Z^n$. Specifically, for $w \in \C^\times$ and $\bs{z} \in V$ with
coordinates $\bs{z} = (z_1,z_2,\ldots,z_n)$, we have
\[
    w\cdot(z_1,z_2,\ldots,z_n) = (w^{a_1} z_1, w^{a_2} z_2, \ldots, w^{a_n} z_n).
\]
To indicate the representation explicitly, we use the notation $\C[V]^{\C^\times_{\bs{a}}}$ to
denote the algebra of polynomials in $z_1,\ldots,z_n$ that are
invariant under this $\C^\times$-action. Similarly, let
$\Hilb_{\bs{a}}(t):= \Hilb_{\C[V]^{\C^\times_{\bs{a}}}}(t)$ denote the Hilbert series of
$\C[V]^{\C^\times_{\bs{a}}}$.

Recall that a representation $V$ is \emph{stable} if it contains an open set consisting of closed
orbits. If $V$ is not stable, then we may replace $V$ with a stable sub-representation without
changing the invariants \cite{WehlauPopov}, so we will assume stability with no loss of generality.
Similarly, the existence of a trivial subrepresentation has the trivial effect of multiplying
$\Hilb_{\bs{a}}(t)$ by $1/(1 - t)$, so we will frequently assume with no loss of generality that
$V^{\C_{\bs{a}}^\times}= \{0\}$, i.e. that $0$ does not appear as a weight. Finally, by taking the quotient
of $\C^\times$ by the kernel of the action on $V$, which does not change the invariants, we often
assume that $V$ is faithful.

It is easy to see that the $\C^\times$-module $V$ with weight vector $\bs{a}$ is stable if and only if $\bs{a}$
contains both positive and negative weights. Let $k$ be the number of negative weights; we will assume the weights
are ordered such that $a_i < 0$ for $i \leq k$ and $a_i > 0$ for $i > k$. The hypothesis that
$V$ is faithful corresponds to $\gcd(a_1,\ldots,a_n) = 1$.

We say the weight vector $\bs{a}$ is \emph{generic} if $a_i \neq a_j$ for $i \neq j$ with
$i, j \leq k$ and \emph{degenerate} otherwise. That is, degenerate weight vectors are those
that have repeated negative weights. Note that these properties are not invariants of the
representation; the weight vectors $\bs{a}$ and $-\bs{a}$ correspond to equivalent representations
though it is clearly possible that $\bs{a}$ is generic while $-\bs{a}$ is degenerate.

By the Molien-Weil formula \cite[Equation (4.6.2)]{DerskenKemperBook}, the Hilbert series
$\Hilb_{\bs{a}}(t)$ is given by
\begin{equation}
\label{eq:MolienWeil}
    \Hilb_{\bs{a}}(t) = \int\limits_{z\in\Sp^1} \frac{dz}{z \det_V (1 - t z)}.
\end{equation}
This will be the starting point of our computations, which follow the idea of
\cite[Section 4.6.4]{DerskenKemperBook}. As explained in the introduction, these computations
were conducted for \emph{cotangent-lifted} representations, those with weight vectors of the
form $(\bs{a}, -\bs{a})$, in \cite{HerbigSeatonHilbSympCirc}; though the regular functions on
the symplectic quotient were the focus of that paper, the computations corresponding to the
usual invariant ring are equivalent by \cite[Proposition 2.1]{HerbigSeatonHilbSympCirc}.
By \cite[Proposition (5.8)]{GWSLiftingHomotopies}, up to tensoring with $\C$, the invariant ring
corresponding to $(\bs{a}, -\bs{a})$ is equal to the ring of real invariants of the action with weight vector
$\bs{a}$. Hence, this paper constitutes a generalization of these results to arbitrary representations.

If $R$ is a Cohen-Macaulay ring, then the $a$-invariant $a(R)$ was defined in
\cite[Definition 3.1.4]{GotoWatanabe} to be the negative of the least degree of a generator
of the canonical module of $R$. This is equal to the degree of $\Hilb_R(t)$, i.e. the
degree of the numerator minus the degree of the denominator, and this is sometimes
taken as an extension of the definition of the $a$-invariant to any positively-generated
algebra over a field; see \cite[Theorem 4.4.3 and Definition 4.4.4]{BrunsHerzog}.
If $R$ is a Cohen-Macaulay normal domain, then by \cite[Theorem 4.4]{Stanley}, $R$ is
Gorenstein if and only if
\begin{equation}
\label{eq:StanleyGoren}
    \Hilb_R(1/t) = (-1)^{\dim R} t^{-a(R)}\Hilb_R(t).
\end{equation}
If $R$ is Gorenstein, then
\begin{equation}
\label{eq:AInvGammas}
    \frac{2\gamma_1(R)}{\gamma_0(R)} = -a(R) - \dim(R)
\end{equation}
where the $\gamma_m(R)$ denote the Laurent coefficients of $\Hilb_R(t)$ as in
Equation~\eqref{eq:DefGammas}. See \cite[Equation (3.32)]{PopovVinberg}, and note
that $q$ in that reference denotes $-a(R)$.


\section{Computation of the Hilbert series}
\label{sec:HilbSer}

In this section, we compute the Hilbert series $\Hilb_{\bs{a}}(t)$ of the invariants associated
to a weight vector $\bs{a}$. We first consider the case that $\bs{a}$ is generic in
Proposition~\ref{prop:HilbSumGeneric}. Though this is a special case of Theorem~\ref{thrm:HilbSer} below,
we include the brief proof, as it illustrates the fundamental idea behind the other computations
in this section.

\begin{proposition}
\label{prop:HilbSumGeneric}
Let $\bs{a} = (a_1,\ldots,a_n) \in \Z^n$ be a weight vector for an action of $\C^\times$
on $\C^n$. Assume that $a_i < 0$ for $i\leq k$ and $a_i > 0$ for $i > k$. Assume further that
$\bs{a}$ is generic ($a_i\neq a_j$ for $i\neq j$ and $i,j\leq k$) and stable ($0 < k < n$). Then
\begin{equation}
\label{eq:HilbSumGeneric}
    \Hilb_{\bs{a}}(t)
    =
    \sum\limits_{i=1}^k \sum\limits_{\zeta^{-a_i} = 1}
        \frac{1}{-a_i \prod\limits_{\substack{j=1\\j\neq i}}^n 1 - \zeta^{a_j} t^{(a_i-a_j)/a_i}}.
\end{equation}
\end{proposition}
\begin{proof}
By the Molien--Weyl formula, Equation~\eqref{eq:MolienWeil}, we have
\[
    \Hilb_{\bs{a}}(t)
    =
    \frac{1}{2\pi i}
    \int_{\Sp^1} \frac{dz}{z \prod_{j=1}^n (1 - z^{a_j} t)}
    =
    \frac{1}{2\pi i}
    \int_{\Sp^1} \frac{z^{-1- \sum_{j=1}^k a_j} dz}{\prod\limits_{j=1}^k (z^{-a_j} - t)
        \prod\limits_{j=k+1}^n (1 - z^{a_j}t)}.
\]
From the latter expression, we see that for fixed $t$ with $|t|<1$,
the poles inside the unit circle occur when $z^{-a_j} = t$ for $1\leq j\leq k$.
Fix an $i$ with $1\leq i \leq k$ and an $-a_i$th root of unity $\zeta_0$,
and let $t^{-1/a_i}$ be defined with respect to a fixed, suitably chosen branch of the
$\log$ function. We then express
\begin{align}
    \nonumber
    \Hilb_{\bs{a}}(t)
    &=
    \frac{1}{2\pi i}
    \int_{\Sp^1} \frac{z^{-a_i - 1} dz}{ (z^{-a_i} - t)
        \prod\limits_{\substack{j=1 \\ j\neq i}}^n (1 - z^{a_j} t)}
    \\ \label{eq:GenericResidueDeviation}
    &=
    \frac{1}{2\pi i}
    \int_{\Sp^1} \frac{z^{-a_i-1} dz}{ (z - \zeta_0 t^{-1/a_i})
        \prod\limits_{\substack{\zeta^{-a_i}=1 \\ \zeta\neq\zeta_0}} (z - \zeta t^{-1/a_i})
        \prod\limits_{\substack{j=1 \\ j\neq i}}^n (1 - z^{a_j} t)}.
\end{align}
Hence, the residue at $\zeta_0 t^{-1/a_i}$ is given by
\begin{align*}
    \Res\limits_{z=\zeta_0 t^{-1/a_i}} f(z,t)
    &=
    \frac{(\zeta_0 t^{-1/a_i})^{-a_i-1}}{
        \prod\limits_{\substack{\zeta^{-a_i}=1 \\ \zeta\neq\zeta_0}}
        (\zeta_0 t^{-1/a_i} - \zeta t^{-1/a_i})
        \prod\limits_{\substack{j=1 \\ j\neq i}}^n (1 - t (\zeta_0 t^{-1/a_i})^{a_j})}
    \\&=
    \frac{1}{ -a_i
        \prod\limits_{\substack{j=1 \\ j\neq i}}^n (1 - \zeta_0^{a_j} t^{(a_i-a_j)/a_i})},
\end{align*}
where we simplify using
\[
    \prod\limits_{\substack{\zeta^{-a_i}=1 \\ \zeta\neq\zeta_0}} (\zeta_0 t^{-1/a_i} - \zeta t^{-1/a_i})
    =
    (\zeta_0t^{-1/a_i})^{-a_i-1}  \prod\limits_{\substack{\zeta^{-a_i}=1 \\ \zeta\neq 1}} (1 - \zeta)
    =
    \zeta_0^{-1} t^{(a_i+1)/a_i}(-a_i).
\]
Summing the residues over all choices of $i \leq k$ and $\zeta_0$, yields the claim.
\end{proof}

We now consider the same computation in the case that $\bs{a}$ is degenerate. The change occurs
in the decomposition of the integrand given in Equation~\eqref{eq:GenericResidueDeviation}. If
$a_i = a_j$ for some $j$ (which must be $\leq k$), then the factor $(1 - z^{a_j}t)$ also vanishes
at $z = \zeta_0 t^{-1/a_i}$, which is no longer a simple pole of the integrand.

We give two ways of approaching this computation. The first, yielding Proposition~\ref{prop:HilbSumDegen1},
is simply to compute each such residue. This results in Equation~\eqref{eq:HilbSumDegen1} which, while not
particularly elucidating, is useful for calculations of specific examples. The second method, yielding
Theorem~\ref{thrm:HilbSer}, is to introduce new variables so that each pole remains simple. This result
essentially demonstrates that, treating the weights $a_i$ in Equation~\eqref{eq:GenericResidueDeviation}
as real variables except in choosing the root of unity $\zeta$, the apparent singularities that occur in
factors of the form $1 - \zeta^{a_j} t^{(a_i-a_j)/a_i}$ when $\zeta = 1$ and $a_i = a_j$ are in fact
removable, and the Hilbert series at degeneracies is the extension of this continuous function of the
weights at these removable singularities.

We begin with the first method. To simplify the argument, we use a slightly different notation to index
the weights.

\begin{proposition}
\label{prop:HilbSumDegen1}
Let $(\bs{a},\bs{b})$ be a weight vector for an action of $\C^\times$ on $\C^n$, where
\[
    \bs{a}
    =   (\overbrace{a_1,\ldots,a_1}^{r_1},
        \overbrace{a_2,\ldots,a_2}^{r_2},
        \ldots,
        \overbrace{a_k,\ldots,a_k}^{r_k}),
    \quad\quad
    \bs{b}
    =    (b_1, \ldots, b_m),
\]
each $r_i \geq 1$, and $n = m+\sum_{i=1}^k r_i$.
We assume that each $a_i < 0$, that each $b_i > 0$, and that the representation is stable ($k > 0$ and $m > 0$);
note that the $b_i$ need not be distinct. Then
\begin{equation}
\label{eq:HilbSumDegen1}
    \Hilb_{(\bs{a},\bs{b})}(t)
    =
    \sum\limits_{i=1}^k \sum\limits_{\zeta_0^{-a_i} = 1}
        \left. \frac{\partial^{r_i-1}}{\partial z^{r_i-1}}\right|_{z=\zeta_0 t^{-1/a_i}}
        \frac{1}{(r_i-1)!}F_{i,\bs{a},\bs{b}}(z, t),
\end{equation}
where
\begin{equation}
\label{eq:TermFunctionDegen1}
    F_{i,\bs{a},\bs{b}}(z, t) =
    \frac{z^{-r_i a_i - 1}}{
        \prod\limits_{\substack{\zeta^{-a_i}=1 \\ \zeta\neq\zeta_0}} (z - \zeta t^{-1/a_j})^{r_i}
        \prod\limits_{\substack{j=1 \\ j\neq i}}^k (1 - z^{a_j} t)
        \prod\limits_{j=1}^m (1 - z^{b_j}t)}.
\end{equation}
\end{proposition}
\begin{proof}
By the Molien--Weyl formula,
\begin{align*}
    \Hilb_{(\bs{a},\bs{b})}(t)
    &=
    \frac{1}{2\pi i}
    \int_{\Sp^1} \frac{dz}{z \prod_{j=1}^k (1 - z^{a_j}t)^{r_j} \prod\limits_{j=1}^m (1 - z^{b_j}t)}
    \\&=
    \frac{1}{2\pi i}
    \int_{\Sp^1} \frac{z^{-1 - \sum_{j=1}^k r_j a_j} dz}{\prod\limits_{j=1}^k (z^{-a_j} - t)^{r_j}
        \prod\limits_{j=1}^m (1 - z^{b_j}t)}.
\end{align*}
As in the proof of Proposition~\ref{prop:HilbSumGeneric}, for fixed $t$ with $|t|<1$,
the poles inside the unit circle occur when $z^{-a_j} = t$ for some $j$. Fix an $i$ with
$1\leq i \leq k$ and an $-a_i$th root of unity $\zeta_0$, and let $t^{-1/a_i}$ be defined
with respect to a suitable fixed branch of the $\log$ function. Then we can express
$\Hilb_{(\bs{a},\bs{b})}(t)$ as
\[
    \frac{1}{2\pi i}
    \int_{\Sp^1} \frac{z^{-r_i a_i - 1} \, dz}{ (z - \zeta_0 t^{-1/a_i})^{r_i}
        \prod\limits_{\substack{\zeta^{-a_i}=1 \\ \zeta\neq\zeta_0}} (z - \zeta t^{-1/a_i})^{r_i}
        \prod\limits_{\substack{j=1 \\ j\neq i}}^k (1 - z^{a_j}t)^{r_j}
        \prod\limits_{j=1}^m (1 - z^{b_j}t)}.
\]
Noting that the factor $F_{i,\bs{a},\bs{b}}(z, t)$ of the integrand, as defined in
Equation~\eqref{eq:TermFunctionDegen1},
is holomorphic at $t = \zeta_0 t^{-1/a_i}$, the residue at $t = \zeta_0 t^{-1/a_i}$ is given by
the $(r_i-1)$st coefficient of the Taylor series of this function at $z = \zeta_0 t^{-1/a_i}$,
i.e.
\[
    \Res\limits_{z=\zeta_0 t^{-1/a_i}} f(z,t)
    =
    \left. \frac{\partial^{r_i-1}}{\partial z^{r_i-1}}\right|_{z=\zeta_0 t^{-1/a_i}}
        \frac{1}{(r_i-1)!}F_{i,\bs{a},\bs{b}}(z, t).
\]
Summing over each pole completes the proof.
\end{proof}

We now turn to the second method for dealing with degeneracies and prove the following,
the main result of this section.

\begin{theorem}
\label{thrm:HilbSer}
Let $\bs{a} = (a_1,\ldots,a_n) \in \Z^n$ be a weight vector for an action of $\C^\times$
on $\C^n$. Assume that $a_i < 0$ for $i\leq k$ and $a_i > 0$ for $i > k$, and moreover that
$\bs{a}$ is stable ($0 < k < n$). Then
\begin{equation}
\label{eq:HilbSumDegen2Unif}
    \Hilb_{\bs{a}}(t)
    =
    \lim\limits_{\bs{c}\to \bs{a}}
    \sum\limits_{i=1}^k \sum\limits_{\zeta^{-a_i} = 1}
        \frac{1}{-c_i \prod\limits_{\substack{j=1\\j\neq i}}^n 1 - \zeta^{a_j} t^{(c_i-c_j)/c_i}},
\end{equation}
$\bs{c}=(c_1,\ldots,c_n)\in \R^n$.
\end{theorem}
Note that we may consider $\bs{c} = (c_1,\ldots,c_k,a_{k+1},\ldots,a_n)$, i.e. take $c_i = a_i$ for $i > k$
as, in the computation below, we require $c_i \neq c_j$ for $i\neq j$ only when $i,j\leq k$.
\begin{proof}
We consider the computation of the residue at a point satisfying $t = z^{-a_i}$ where $a_i<0$ is
a degenerate weight. For simplicity, we relabel the weights to express the weight matrix
as $(a, \ldots, a, a_1, \ldots, a_{k-q}, b_1,\ldots, b_m)$ where $a < 0$ occurs $q$ times, $a_i < 0$
for each $i$, and $a \neq a_i$ for each $i$. We do not require that the $a_i$ nor the $b_i$ are
distinct. Consider the integral
\begin{equation}
\label{eq:HilbSerDegen2Integral}
    \frac{1}{2\pi i}
    \int_{\Sp^1} \frac{dz}{z \prod\limits_{j=1}^q (1 - z^{a} x_j) \prod\limits_{j=1}^{k-q} (1 - z^{a_j} y_j)
    \prod\limits_{j=1}^m (1 - z^{b_j} w_j)}.
\end{equation}
Here, $|x_j|, |y_j|, |w_j| < 1$ and the $x_j$ are assumed distinct. Let
$\bs{x} = (x_1,\ldots,x_q)\in\C^q$,
\linebreak
$\bs{y} = (y_1,\ldots,y_{k-q})\in\C^{k-q}$, and
$\bs{w} = (w_1,\ldots,w_m)\in\C^m$. These additional variables are introduced to force
each pole to be simple; the idea is to compute this integral and then take the limit
as $\bs{x} \to \Delta_q t:= (t, \ldots, t) \in \C^q$,
$\bs{y} \to \Delta_{k-q} t:= (t, \ldots, t) \in \C^{k-q}$, and
$\bs{w} \to \Delta_m t:= (t, \ldots, t) \in \C^m$.

Fix a branch of the logarithm near $t$. We will assume throughout that each $x_i$, $y_i$, and $w_i$
is contained in the domain of this branch. Let
\[
    F(\bs{x},\bs{y},\bs{w},z)
    :=
    \frac{1}{z \prod\limits_{j=1}^q (1 - z^{a} x_j) \prod\limits_{j=1}^{k-q} (1 - z^{a_j} y_j)
    \prod\limits_{j=1}^m (1 - z^{b_j} w_j)}.
\]
As a function of $z$, the poles of $F(\bs{x},\bs{y},\bs{w},z)$ for $|z| < 1$
occur at $z = \zeta x_i^{-1/a}$ where $1 \leq i \leq q$ and $\zeta$ is a $(-a)$th root of unity
or $z = \eta y_i^{-1/a_i}$ where $1 \leq i \leq k-q$ and
$\eta$ is a $(-a_i)$th root of unity. Hence, the integral in Equation~\eqref{eq:HilbSerDegen2Integral}
is given by
\[
    \sum\limits_{\zeta^a=1} \sum\limits_{i=1}^q
    \Res\limits_{z=\zeta x_i^{-1/a}} F(\bs{x},\bs{y},\bs{w},z)
    +
    \sum\limits_{i=1}^{k-q} \sum\limits_{\eta^{a_i}=1}
    \Res\limits_{z=\eta y_i^{-1/a_i}} F(\bs{x},\bs{y},\bs{w},z).
\]

Fix a $(-a)$th root of unity $\zeta_0$. We claim that
\[
    R_a(\bs{x},\bs{y},\bs{w},\zeta_0)
    := \sum\limits_{i=1}^q\Res\limits_{z=\zeta_0 x_i^{-1/a}} F(\bs{x},\bs{y},\bs{w},z)
\]
admits an analytic continuation whose domain includes $x_1=x_2=\cdots=x_q=t$.
To see this, fix an $i$, and then we express $F(\bs{x},\bs{y},\bs{w},z)$ as
\[
    \frac{z^{-a-1}}{(z - \zeta_0 x_i^{-1/a})
    \prod\limits_{\substack{\zeta^{-a}=1\\ \zeta\neq\zeta_0}} (z - \zeta x_i^{-1/a})
    \prod\limits_{\substack{j=1 \\ j\neq i}}^q (1 - z^{a} x_j)
    \prod\limits_{j=1}^{k-q} (1 - z^{a_j} y_j)
    \prod\limits_{j=1}^m (1 - z^{b_j} w_j)}.
\]
Hence we have a simple pole at $\zeta_0 x_i^{-1/a}$. Computing the reside and simplifying
following the same steps as in the proof of Proposition~\ref{prop:HilbSumGeneric}, we express
the residue $\Res_{z= \zeta_0 x_i^{-1/a}} F(\bs{x},\bs{y},\bs{w},z)$ as
\[
    \frac{ x_i^{q-1} }{-a
    \prod\limits_{\substack{j=1 \\ j\neq i}}^q (x_i - x_j)
    \prod\limits_{j=1}^{k-q} (1 - \zeta_0^{a_j} x_i^{-a_j/a} y_j)
    \prod\limits_{j=1}^m (1 - \zeta_0^{b_j} x_i^{-b_j/a} w_j)}.
\]
Summing over $i=1,\ldots,q$ and combining the result into a single rational expression,
we express $R_a(\bs{x},\bs{y},\bs{w},\zeta_0) = \sum_{i=1}^q \Res_{z= \zeta_0 x_i^{-1/a}} F(\bs{x},\bs{y},\bs{w},z)$
as
\[
    \frac{ \sum\limits_{i=1}^q (-1)^{q-i} x_i^{q-1}
    \prod\limits_{\substack{1 \leq j < \ell \leq q \\ j, \ell \neq i}} (x_\ell - x_j)
    \prod\limits_{\substack{\ell=1\\ \ell\neq i}}^q \prod\limits_{j=1}^{k-q} (1 - \zeta_0^{a_j} x_\ell^{-a_j/a} y_j)
    \prod\limits_{\substack{\ell=1\\ \ell\neq i}}^q \prod\limits_{j=1}^m (1 - \zeta_0^{b_j} x_\ell^{-b_j/a} w_j)
    }{-a
    \prod\limits_{1 \leq j < \ell \leq q} (x_\ell - x_j)
    \prod\limits_{\ell=1}^q \prod\limits_{j=1}^{k-q} (1 - \zeta_0^{a_j} x_\ell^{-a_j/a} y_j)
    \prod\limits_{\ell=1}^q \prod\limits_{j=1}^m (1 - \zeta_0^{b_j} x_\ell^{-b_j/a} w_j)}.
\]
We claim that the numerator of $R_a(\bs{x},\bs{y},\bs{w},\zeta_0)$ is alternating in the $x_1,\ldots,x_q$.
To see this, define
\begin{align*}
    &\alpha_i(\bs{x},\bs{y},\bs{w},\zeta_0):=
    \\&\quad
    (-1)^{q-i} x_i^{q-1}
    \prod\limits_{\substack{1 \leq j < \ell \leq q \\ j, \ell \neq i}} (x_\ell - x_j)
    \prod\limits_{\substack{\ell=1\\ \ell\neq i}}^q \prod\limits_{j=1}^{k-q} (1 - \zeta_0^{a_j} x_\ell^{-a_j/a} y_j)
    \prod\limits_{\substack{\ell=1\\ \ell\neq i}}^q \prod\limits_{j=1}^m (1 - \zeta_0^{b_j} x_\ell^{-b_j/a} w_j)
\end{align*}
so that the numerator of $R_a(\bs{x},\bs{y},\bs{w},\zeta_0)$ is equal to $\sum_{i=1}^q \alpha_i(\bs{x},\bs{y},\bs{w},\zeta_0)$.
It is easy to see that for a transposition
$\sigma\in\mathcal{S}_q$,
$\alpha_i(x_{\sigma(1)},\ldots,x_{\sigma{q}},\bs{y},\bs{w},\zeta_0) = -\alpha_{\sigma(i)}(x_1,\ldots,x_q,\bs{y},\bs{w},\zeta_0)$,
implying that the numerator of $R_a(\bs{x},\bs{y},\bs{w},\zeta_0)$ is alternating. It follows that there is a
$S(\bs{x},\bs{y},\bs{w},\zeta_0)$, symmetric in the $x_1,\ldots,x_q$, such that $R_a(\bs{x},\bs{y},\bs{w},\zeta_0)$
is equal to
\[
    \frac{ S(\bs{x},\bs{y},\bs{w},\zeta_0)}
    {a
    \prod\limits_{\ell=1}^q \prod\limits_{j=1}^{k-q} (1 - \zeta_0^{a_j} x_\ell^{-a_j/a} y_j)
    \prod\limits_{\ell=1}^q \prod\limits_{j=1}^m (1 - \zeta_0^{b_j} x_\ell^{-b_j/a} w_j)}.
\]
Hence, as long as the $y_j$ are chosen so that $\zeta_0^{a_j} x_\ell^{-a_j/a} y_j \neq 1$,
the singularities at $x_i = x_j$ are removable.

Now that we have determined that the limit $\lim_{\bs{x}\to\Delta_q t} R_a(\bs{x},\bs{y},\bs{w},\zeta_0)$
exists, we will carry out a sequence of parameterizations to compute its value.
Let $\bs{t}=(t_1,\ldots,t_q)$ and $\bs{s}=(s_1,\ldots,s_q)$
and express $x_j = t^{t_j}$, $y_j = t$, and $w_j = t$ to express
$\lim_{\bs{x}\to\Delta_q t} R_a(\bs{x},\bs{y},\bs{w},\zeta_0)$ as
\begin{align*}
    &
    \lim\limits_{\bs{t}\to\Delta_q 1}
    \sum\limits_{i=1}^q \frac{ 1 }{-a
    \prod\limits_{\substack{j=1 \\ j\neq i}}^q (1 - t^{t_j-t_i})
    \prod\limits_{j=1}^{k-q} (1 - \zeta_0^{a_j} t^{(a-t_i a_j)/a})
    \prod\limits_{j=1}^m (1 - \zeta_0^{b_j} t^{(a-t_i b_j)/a})}
    \\&\quad =
    \lim\limits_{\substack{\bs{t}\to\Delta_q 1 \\ \bs{s}\to\Delta_q 1}}
    \sum\limits_{i=1}^q \frac{ 1 }{\frac{-a}{s_i}
    \prod\limits_{\substack{j=1 \\ j\neq i}}^q (1 - t^{(t_j-t_i)/s_j})
    \prod\limits_{j=1}^{k-q} (1 - \zeta_0^{a_j} t^{(a-t_i a_j)/a})
    \prod\limits_{j=1}^m (1 - \zeta_0^{b_j} t^{(a-t_i b_j)/a})}
    \\&\quad =
    \lim\limits_{\bs{t}\to\Delta_q 1}
    \sum\limits_{i=1}^q \frac{ 1 }{\frac{a}{t_i}
    \prod\limits_{\substack{j=1 \\ j\neq i}}^q (1 - t^{(t_j-t_i)/t_j})
    \prod\limits_{j=1}^{k-q} (1 - \zeta_0^{a_j} t^{(a-t_i a_j)/a})
    \prod\limits_{j=1}^m (1 - \zeta_0^{b_j} t^{(a-t_i b_j)/a})},
\end{align*}
where in the last step we set $s_i = t_i$ for each $i$. Doing some elementary
algebra and setting $c_i = a/t_i$ for $i = 1,\ldots,q$, we obtain
\[
    \lim\limits_{\bs{c}\to\Delta_q a}
    \sum\limits_{i=1}^q \frac{ 1 }{c_i
    \prod\limits_{\substack{j=1 \\ j\neq i}}^q (1 - t^{(c_i-c_j)/c_i})
    \prod\limits_{j=1}^{k-q} (1 - \zeta_0^{a_j} t^{(c_i-a_j)/c_i})
    \prod\limits_{j=1}^m (1 - \zeta_0^{b_j} t^{(c_i - b_j)/c_i})}.
\]
Applying this computation to each of the poles and returning to the notation for $\bs{a}$
in the statement of the theorem, we obtain that the sum of the residues of all poles in the
unit circle is given by
\[
    \sum\limits_{i=1}^k \sum\limits_{\zeta^{a_i} = 1}
        \frac{1}{c_i \prod\limits_{\substack{j=1\\j\neq i}}^k (1 - \zeta^{a_j} t^{(c_i-c_j)/c_i})
        \prod\limits_{j=1}^m (1 - \zeta^{b_j}t^{(c_i-b_j)/c_i} )}.
\]
It remains only to show that we may exchange the limit with the integral. But this follows from
a simple application of the Dominated Convergence Theorem, where we note that the integrand is
continuous and hence bounded on the circle.
\end{proof}


\section{An algorithm to compute the Hilbert series in the generic case}
\label{sec:DirtyMethod}

In this section, we outline how Proposition~\ref{prop:HilbSumGeneric} can
be turned into an algorithm to compute the Hilbert series $\Hilb_{\bs{a}}(t)$
associated to a generic weight vector. This algorithm is very similar to that
described in \cite[Section 4]{HerbigSeatonHilbSympCirc}, so we give a brief summary
and refer the reader to that reference for more details.

Recall that for a formal
power series $G(t) = \sum_{m=0}^\infty G_m t^m$ and a positive integer $a$,
the operator $U_a$ is given by
\[
    (U_a G)(t) := G_{(a)}(t) := \sum\limits_{m=0}^\infty G_{ma} t^m.
\]
By \cite[Lemma 4.1]{HerbigSeatonHilbSympCirc}, if $G(t)$ is the power series of a rational
function, then $G_{(a)}(t)$ is as well rational.

Now, start with a generic weight vector $(a_1,\ldots,a_n)$ with $a_i < 0$ for $i \leq k$
and $a_i > 0$ for $i > k$. For $i=1,\ldots,k$, define
\[
    \widetilde{\Phi_i}(t)
    :=
    \frac{1}{\prod\limits_{\substack{j=1\\j\neq i}}^n 1 - t^{(a_i-a_j)/a_i}},
\]
which we note is analytic at $t=0$. Using the expression of $U_a$ in terms of averaging over
roots of unity described in \cite[Section 4]{HerbigSeatonHilbSympCirc},
Equation~\eqref{eq:HilbSumGeneric} can be written as
\[
    \Hilb_{\bs{a}}(t)
    =
    \sum\limits_{i=1}^k \frac{1}{-a_i}
        \sum\limits_{\zeta^{-a_i} = 1} \widetilde{\Phi_i}(\zeta t)
    =
    \sum\limits_{i=1}^k (\widetilde{\Phi_i})_{(-a_i)}(t).
\]
Expressing $\widetilde{\Phi_i}(t)$ as a rational function $P(t)/Q(t)$ where
$P(t)$ is a monomial and $Q(t)$ is a product of terms of the form
$(1 - t^s)$, we replace each $(1 - t^s)$ in the denominator with
\[
    (1 - t^{\lcm{(-a_i,s)}})^{\gcd{(-a_i,s)}}.
\]
This yields the denominator of $(\widetilde{\Phi_i})_{(a_i)}(t)$.
Then we determine the Taylor series of $(\widetilde{\Phi_i})_{(a_i)}(t)$ using
the $U_{(a_i)}$ operator and multiply these out to determine the numerator of
$(\widetilde{\Phi_i})_{(a_i)}(t)$. By Kempf's bound \cite[Theorem 4.3]{KempfHochRob},
$\dim P(t) \leq \dim Q(t)$ so that we need only compute the Taylor series of
$(\widetilde{\Phi_i})_{(a_i)}(t)$ up to degree $\deg(Q)$. Then the Hilbert
series is given by the sum of the $(\widetilde{\Phi_i})_{(a_i)}(t)$ for
$1 \leq i \leq k$. This algorithm has been implemented on \emph{Mathematica}
\cite{Mathematica} and is available from the authors upon request.


\section{Partial Laurent-Schur polynomials}
\label{sec:Schur}

In this section, we introduce a family of polynomials, independently symmetric in two sets of variables,
that will be useful for describing the Laurent coefficients of the Hilbert series of circle invariants
and, in particular, removing singularities in their descriptions.

For a set of $k$ indeterminates $\bs{x} = (x_1,\ldots,x_k)$, we let $\Vand(\bs{x})$ denote the Vandermonde
determinant
\[
    \Vand(\bs{x}) = \prod_{1\leq i < j \leq k} (x_i - x_j).
\]
We now state the following.

\begin{definition}
\label{def:PartialSchur}
Let $\bs{x} = (x_1,\ldots,x_k)$ and $\bs{y} = (y_1,\ldots,y_m)$ be two sets of indeterminates with
$n = k + m$, and let $u \leq n-2$ be an integer. We define the \emph{partial Laurent-Schur polynomial}
$S_u(x_1,\ldots,x_k,y_1,\ldots,y_m) = S_u(\bs{x},\bs{y})$ to be
\begin{equation}
\label{eq:DefS}
    S_u(\bs{x},\bs{y}) =
    \frac{1}{\Vand(\bs{x})\Vand(\bs{y})}
    \begin{vmatrix}
    x_1^u      & \cdots & x_k^u      & 0             & \cdots & 0            \\
    x_1^{n-2}  & \cdots & x_k^{n-2}  & y_{1}^{n-2} & \cdots & y_m^{n-2}    \\
    x_1^{n-3}  & \cdots & x_k^{n-3}  & y_{1}^{n-3} & \cdots & y_m^{n-3}    \\
    \vdots     &        & \vdots     & \vdots        &        & \vdots       \\
    x_1        & \cdots & x_k        & y_{1}       & \cdots & y_m          \\
    1          & \cdots & 1          & 1             & \cdots & 1
    \end{vmatrix}.
\end{equation}
\end{definition}

To indicate the connection with ordinary Schur polynomials, let $\lambda$ be a \emph{partition}
of length $n$, i.e. $\lambda\in\Z^n$ with $\lambda_1\geq\lambda_2\geq\cdots\geq\lambda_n \geq 0$.
Recall that the \emph{alternant} associated to $\lambda$ in the indeterminates $\bs{x} = (x_1,\ldots,x_n)$
is the determinant
\[
    A_\lambda (\bs{x})
    =
    A_{\lambda_1,\ldots,\lambda_n} (x_1,\ldots,x_n)
    =
    \begin{vmatrix}
        x_1^{\lambda_1} &   x_2^{\lambda_1} &   \cdots  &   x_n^{\lambda_1}
        \\
        x_1^{\lambda_2} &   x_2^{\lambda_2} &   \cdots  &   x_n^{\lambda_2}
        \\
        \vdots          &   \vdots          &           &   \vdots
        \\
        x_1^{\lambda_n} &   x_2^{\lambda_n} &   \cdots  &   x_n^{\lambda_n}
        \\
    \end{vmatrix}.
\]
The \emph{Schur polynomial} associated to $\lambda$ is the symmetric polynomial
\[
    s_\lambda( \bs{x} )
        =   \frac{ A_{\lambda + \delta_n} ( \bs{x}) }{\Vand(\bs{x}) }
        =   \frac{ A_{\lambda + \delta_n} ( \bs{x}) }
                {A_{\delta_n} ( \bs{x}) },
\]
where $\delta_n = (n-1,n-2,\ldots,0)$.
The fact that $A_\lambda (\bs{x})$ is obviously alternating in the $x_i$ implies that the polynomial
$A_\lambda ( \bs{x})$ is divisible by $\Vand(\bs{x})$ and hence $s_\lambda( \bs{x} )$
is a polynomial. If we allow the $\lambda_i$ to take negative values, then $\lambda$
is called a \emph{signature}, and $s_\lambda( \bs{x} )$, defined in the same way,
is the \emph{Laurent-Schur polynomial} associated to $\lambda$.
See \cite[I.3]{MacdonaldSymFuncs} or \cite[4.4--6]{SaganSymGrp} for more details.

We now indicate the motivation for Definition~\ref{def:PartialSchur}.
In the computations of the Laurent coefficients of the Hilbert series $\Hilb_{\bs{a}}(t)$
of circle invariants in Section~\ref{sec:Laurent}, we will frequently run into rational
functions of the form
\begin{equation}
\label{eq:PartSchurMotiv}
    \frac{\sum\limits_{i=1}^k (-1)^{i-1}x_i^{u}
        \prod\limits_{\substack{1\leq p < q \leq n\\ p,q\neq i}} (x_p-x_q)}
        {\prod\limits_{1\leq p < q \leq n} (x_p-x_q)}
    =
    \frac{\sum\limits_{i=1}^k (-1)^{i-1}x_i^{u}
        \prod\limits_{\substack{1\leq p < q \leq n\\ p,q\neq i}} (x_p-x_q)}
        {\Vand(x_1,\ldots,x_n)}
\end{equation}
where $u \leq n-2$ is an integer. Using $\widehat{\ast}$ to denote removed columns, we can express
\[
    \sum\limits_{i=1}^k (-1)^{i-1}x_i^{u}
        \prod\limits_{\substack{1\leq p < q \leq n\\ p,q\neq i}} (x_p-x_q)
    =
    \sum\limits_{i=1}^k (-1)^{i-1}x_i^{u}
    \begin{vmatrix}
    x_1^{n-2} & \cdots & \widehat{x_i^{n-2}} & \cdots & x_n^{n-2}  \\
    x_1^{n-3} & \cdots & \widehat{x_i^{n-3}} & \cdots & x_n^{n-3}  \\
    \vdots    &        & \vdots              &        & \vdots     \\
    x_1       & \cdots & \widehat{x_i} & \cdots & x_n        \\
    1         & \cdots & \widehat{1}         & \cdots & 1
    \end{vmatrix}
\]
to see that the numerator of Equation~\eqref{eq:PartSchurMotiv} is the first $k$ terms of the cofactor
expansion of the alternant
$A_{u,n-2,n-3,\ldots,0} (x_1,\ldots,x_n)$ along the first row
which again coincides with the determinant
\begin{equation}
\label{eq:DefS01}
    \begin{vmatrix}
    x_1^u      & \cdots & x_k^u      & 0             & \cdots & 0            \\
    x_1^{n-2}  & \cdots & x_k^{n-2}  & x_{k+1}^{n-2} & \cdots & x_n^{n-2}    \\
    x_1^{n-3}  & \cdots & x_k^{n-3}  & x_{k+1}^{n-3} & \cdots & x_n^{n-3}    \\
    \vdots     &        & \vdots     & \vdots        &        & \vdots       \\
    x_1        & \cdots & x_k        & x_{k+1}       & \cdots & x_n          \\
    1          & \cdots & 1          & 1             & \cdots & 1
    \end{vmatrix}.
\end{equation}
Relabeling variables to account
for the missing symmetries, it is obvious that Equation~\eqref{eq:DefS01} corresponds to the determinant
given in Equation~\eqref{eq:DefS}. While this determinant is not divisible by the full
Vandermonde determinant $\Vand(x_1,\ldots,x_n)$, it is divisible by the partial Vandermonde
determinants in the first $k$ and last $n-k$ variables in Equation~\eqref{eq:DefS}.
The purpose of this section is to give explicit descriptions of the quotient polynomials.
To facilitate this, we introduce the following notation.

For a positive integer $n$, let $\mathbf{n} = \{1,2,\ldots,n\}$ and
$\mathbf{\underline{n}} = \{ 0,1,2,\ldots, n\}$. If
$\lambda$ is a partition with $\lambda_1 \leq n$, we let $\lambda^{c,n}$ denote the complement
of $\lambda$ in $\mathbf{\underline{n}}$, i.e. the partition that contains one instance of each
element of $\mathbf{\underline{n}}$ that does not appear in $\lambda$. If $\lambda$ is a partition
with $\lambda_1 \leq n$ containing the integer $u$, let $\lambda_u^{c,n}$ denote
the partition $\lambda^{c,n}$ with the entry $u$ added. We then have the following.

\begin{theorem}
\label{thrm:PartSchur}
Let $\bs{x} = (x_1,\ldots,x_k)$ and $\bs{y} = (y_1,\ldots,y_m)$ be two sets of indeterminants with
$n = k + m$, and let $u \leq n-2$ be an integer. If $u\geq 0$, then
\begin{equation}
\label{eq:PartSchurPositive}
    S_u(\bs{x},\bs{y})
    =
    (-1)^{k(k+1)/2 + n(k-1) + u}
        \hspace{-.3cm}
        \sum\limits_{\substack{n-2\geq \lambda_1>\cdots>\lambda_k\geq 0\\ \lambda_i = u}}
        \hspace{-.3cm}
        (-1)^{\|\lambda\|+i}
        s_{\lambda-\delta_k} (\bs{x})
        s_{\lambda_u^{c,n-2}-\delta_m} (\bs{y}),
\end{equation}
where the sum is over all partitions $\lambda$ of length $k$ with $\lambda_1 \leq n-2$ such that
$\lambda_i = u$ for some $i$.

If $u < 0$, then
\begin{equation}
\label{eq:PartSchurNegative}
    S_u(\bs{x},\bs{y})
    =
    (-1)^{k(k-1)/2 + n(k-1)}
        \hspace{-.3cm}
        \sum\limits_{n-2\geq \lambda_1>\cdots>\lambda_{k-1}\geq 0}
        \hspace{-.3cm}
        (-1)^{\|\lambda\|}
        s_{(\lambda,u)-\delta_k} (\bs{x})
        s_{\lambda^{c,n-2}-\delta_m} (\bs{y}),
\end{equation}
where the sum is over all partitions of $\lambda$ of length $k-1$ with $\lambda_1 \leq n-2$.

In either case, $S_u(\bs{x},\bs{y})$ is homogeneous of degree $(m-1)(k-1)+u$.
\end{theorem}

Note that if $u \geq 0$, then via row reduction, it is obvious that we can express
\begin{equation}
\label{eq:DefSRowRed}
    S_u(\bs{x},\bs{y}) =
    \frac{1}{\Vand(\bs{x})\Vand(\bs{y})}
    \begin{vmatrix}
    x_1^u      & \cdots & x_k^u      & 0             & \cdots & 0            \\
    x_1^{n-2}  & \cdots & x_k^{n-2}  & y_{1}^{n-2} & \cdots & y_m^{n-2}    \\
    \vdots     &        & \vdots     & \vdots        &        & \vdots       \\
    x_1^{u+1}  & \cdots & x_k^{u+1}  & y_{1}^{u+1} & \cdots & y_m^{u+1}    \\
    0          & \cdots & 0          & y_{1}^u     & \cdots & y_m^u        \\
    x_1^{u-1}  & \cdots & x_k^{u-1}  & y_{1}^{u-1} & \cdots & y_m^{u-1}    \\
    \vdots     &        & \vdots     & \vdots        &        & \vdots       \\
    1          & \cdots & 1          & 1             & \cdots & 1
    \end{vmatrix},
    \quad   u \geq 0.
\end{equation}

\begin{proof}[Proof of Theorem~\ref{thrm:PartSchur}]
We first consider the case $0 \leq u \leq n-2$. For $P\subseteq\mathbf{\underline{n}}$,
let $|P|$ denote the cardinality of $P$ and $\|P\|$ the sum of the elements of $P$.
Let
\[
    Z =
    \begin{pmatrix}
    x_1^u     & x_2^u     & \cdots & x_k^u      & 0             & \cdots & 0            \\
    x_1^{n-2} & x_2^{n-2} & \cdots & x_k^{n-2}  & y_{1}^{n-2} & \cdots & y_m^{n-2}    \\
    \vdots    & \vdots    &        & \vdots     & \vdots        &        & \vdots       \\
    x_1^{u+1} & x_2^{u+1} & \cdots & x_k^{u+1}  & y_{1}^{u+1} & \cdots & y_m^{u+1}    \\
    0         & 0         & \cdots & 0          & y_{1}^u     & \cdots & y_m^u        \\
    x_1^{u-1} & x_2^{u-1} & \cdots & x_k^{u-1}  & y_{1}^{u-1} & \cdots & y_m^{u-1}    \\
    \vdots    & \vdots    &        & \vdots     & \vdots        &        & \vdots       \\
    x_1       & x_2       & \cdots & x_k        & y_{1}       & \cdots & y_m          \\
    1         & 1         & \cdots & 1          & 1             & \cdots & 1
    \end{pmatrix}
\]
be the matrix whose determinant appears in Equation~\eqref{eq:DefSRowRed}.
For $P,Q\subseteq\mathbf{n}$ such that $|P|=|Q|$, let $Z^{P,Q}$ denote the cofactor
of $Z$ formed by the determinant of the submatrix with rows in $P$ and columns in $Q$.
Note that $Z^{P,\mathbf{k}} = 0$ for each $P\subseteq \mathbf{n}$ with $|P|=k$ that contains
the element $n-u$, and similarly
$Z^{\mathbf{n}\smallsetminus P,\mathbf{n}\smallsetminus\mathbf{k}} = 0$
for each $P\subseteq \mathbf{n}$ with $|P|=k$ that does not contain $1$.
Hence, the Laplace expansion of the determinant along the first $k$ columns,
see \cite[Theorem 13.8.1]{Harville}, is given by
\[
    S_u(\bs{x},\bs{y})
    =
    \frac{1}{\Vand(\bs{x})\Vand(\bs{y})}
    \sum\limits_{\substack{1\in P\subseteq \mathbf{n}, n-u\not\in P\\|P|=k}}
        (-1)^{\|\mathbf{k}\| + \|P\|}
            Z^{P,\mathbf{k}}
            Z^{\mathbf{n}\smallsetminus P,
                \mathbf{n}\smallsetminus\mathbf{k}}.
\]
To express the minors $Z^{P,\mathbf{k}}$ and
$Z^{\mathbf{n}\smallsetminus P,\mathbf{n}\smallsetminus\mathbf{k}}$ in terms of alternants,
consider a choice of $P$ with elements $1=p_1<p_2<\cdots<p_k$. Then
$Z^{P,\mathbf{k}} = (-1)^{i-1} A_\lambda(\bs{x})$ where $\lambda$ is a partition
of length $k$, $n-2\geq\lambda_1 \geq\cdots\geq\lambda_k\geq 0$, and $\lambda_i = u$.
In particular, $\lambda$ is given by transposing the first entry of
$(u,n-p_2,n-p_3,\ldots,n-p_k)$ so that the result is decreasing, hence
$\|P\| = n(k-1) - \|\lambda\| + u + 1$. Similarly,
$Z^{\mathbf{n}\smallsetminus P,\mathbf{n}\smallsetminus\mathbf{k}} = A_{\lambda_u^{c,n-2}} (\bs{y})$,
where $\lambda_u^{c,n-2}$ is the complement of $\lambda$ in $\mathbf{\underline{n-2}}$
with the element $u$ added. Therefore, we can express
\begin{align*}
    S_u(\bs{x},\bs{y})
    &=
    (-1)^{k(k+1)/2 + n(k-1) + u}\hspace{-.3cm}
    \sum\limits_{\substack{n-2\geq \lambda_1>\cdots>\lambda_k\geq 0\\ \lambda_i = u}}
        \hspace{-.3cm}
        (-1)^{\|\lambda\| + i}
            \frac{A_\lambda (\bs{x})}{\Vand(\bs{x})}
            \frac{A_{\lambda_u^{c,n-2}} (\bs{y})}{\Vand(\bs{y})}
    \\&=
    (-1)^{k(k+1)/2 + n(k-1) + u}\hspace{-.3cm}
    \sum\limits_{\substack{n-2\geq \lambda_1>\cdots>\lambda_k\geq 0\\ \lambda_i = u}}
        \hspace{-.3cm}
        (-1)^{\|\lambda\| + i}
            s_{\lambda-\delta_{k}}(\bs{x})s_{{\lambda_u^{c,n-2}}-\delta_{m}}(\bs{y}),
\end{align*}
yielding Equation~\eqref{eq:PartSchurPositive}.

Now suppose $u < 0$ and set
\[
    Z =
    \begin{pmatrix}
    x_1^u     & x_2^u     & \cdots & x_k^u      & 0             & \cdots & 0            \\
    x_1^{n-2} & x_2^{n-2} & \cdots & x_k^{n-2}  & y_{1}^{n-2} & \cdots & y_m^{n-2}    \\
    \vdots    & \vdots    &        & \vdots     & \vdots        &        & \vdots       \\
    x_1       & x_2       & \cdots & x_k        & y_{1}       & \cdots & y_m          \\
    1         & 1         & \cdots & 1          & 1             & \cdots & 1
    \end{pmatrix}.
\]
We again consider the Laplace expansion of $\det Z$ along the first $k$ columns,
noting that
$Z^{\mathbf{\underline{n}}\smallsetminus P,\mathbf{\underline{n}}\smallsetminus\mathbf{\underline{k}}} = 0$
for each $P\subseteq \mathbf{n}$ with $|P|=k$ that does not contain $1$, and express
\[
    S_u(\bs{x},\bs{y})
    =
    \frac{1}{\Vand(\bs{x})\Vand(\bs{y})}
    \sum\limits_{\substack{1\in P\subseteq \mathbf{n}\\|P|=k}}
        (-1)^{\|\mathbf{k}\| + \|P\|}
            Z^{P,\mathbf{k}}
            Z^{\mathbf{n}\smallsetminus P,
                \mathbf{n}\smallsetminus\mathbf{k}}.
\]
Considering a choice of $P$ with elements $1=p_1<p_2<\cdots<p_k$, in this case, $Z^{P,\mathbf{k}} = (-1)^{k-1}A_{(\lambda,u)}(\bs{x})$
where $\lambda = (n-p_2,\ldots,n-p_k)$, and
$\|P\| = n(k-1) - \|\lambda\| + 1$. Thus
\begin{align*}
    S_u(\bs{x},\bs{y})
    &=
    (-1)^{k(k+1)/2 + n(k-1)-k}
    \sum\limits_{n-2\geq \lambda_1>\cdots>\lambda_{k-1}\geq 0}
        (-1)^{\|\lambda\|}
        \frac{A_{(\lambda,u)}(\bs{x})}{\Vand(\bs{x})}
        \frac{A_{\lambda^{c,n-2}}(\bs{y})}{\Vand(\bs{y})}
    \\&=
    (-1)^{k(k-1)/2 + n(k-1)}
    \sum\limits_{n-2\geq \lambda_1>\cdots>\lambda_{k-1}\geq 0}
        (-1)^{\|\lambda\|}
        s_{(\lambda,u)-\delta_k}(\bs{x})
        s_{\lambda^{c,n-2} - \delta_m}(\bs{y}),
\end{align*}
yielding Equation~\eqref{eq:PartSchurNegative}. That $S_u(\bs{x},\bs{y})$ is
homogeneous of degree
\[
(n-2)(n-1)+u-(k-1)k/2-(m-1)m/2 = (m-1)(k-1)+u
\]
is easily observed from Equation~\eqref{eq:DefS},
completing the proof.
\end{proof}

Recall that the Schur polynomial $s_\lambda(\bs{x})$ associated to a partition $\lambda$ of length $n$
in the indeterminates $\bs{x} = (x_1,\ldots,x_n)$ can be represented as
\begin{equation}
\label{eq:SchurTableaux}
    s_\lambda(\bs{x})
    =
    \sum\limits_T \bs{x}^T,
\end{equation}
where the sum ranges over all semistandard Young tableaux $T$ of shape $\lambda$.
We include a similar combinatorial interpretation of Theorem~\ref{thrm:PartSchur}.
First, we consider the case $0\le u\le n-2$.

\begin{corollary}
\label{cor:PartSchurTableaux}
If $0\le u\le n-2$, then
\[
    S_u(\bs{x},\bs{y})
    =
    (-1)^{k(k+1)/2 + n(k-1) + u}
    \sum\limits_T (-1)^{\|\lambda\|+i} (\bs{x},\bs{y})^T,
\]
where the sum is over all tableaux $T$ formed as follows:
\begin{enumerate}
\item   Start with a Young diagram of shape
        $(n-2,n-3,\ldots,u+1,u,u,u-1,\ldots,1,0)$, i.e. $n$ rows consisting of
        each integer from $0$ to $n-2$ with $u$ appearing twice.
\item   Label $k$ of the rows $x$ and the remaining $m=n-k$ rows $y$, where the upper
        $u$ row must be labeled $x$ while the lower $u$ row is labeled $y$.
        Note that $\lambda$ corresponds to the shape of the diagram consisting of the rows labeled $x$
        and $\lambda_u^{c,n-2}$ is the shape of the diagram of rows labeled $y$.
        Define $i$ by letting the $u$ row labeled $x$ appear as the $i$th row labeled $x$ from the top.
\item   Starting from the bottom row of length $0$, for each $j$, delete $j-1$ boxes
        from the $j$th row labeled $x$ and delete $j-1$ boxes from the
        $j$th row labeled $y$.
\item   Fill in the rows labeled $x$ with $x_1,\ldots,x_k$ so that the rows labeled $x$
        form a semistandard Young tableau of shape $\lambda-\delta_k$, and similarly fill in the rows labeled $y$ with
        $y_1,\ldots,y_m$ so that the rows labeled $y$ form a semistandard Young tableau of shape
        $\lambda_u^{c,n-2}-\delta_m$. As usual, $(\bs{x},\bs{y})^T$ denotes the product of the entries of $T$.
\end{enumerate}
\end{corollary}

\begin{proof} This is an immediate consequence of interpreting Equation~\eqref{eq:PartSchurPositive} in the
context of Equation~\eqref{eq:SchurTableaux}. Note in particular, that we have multiplicativity
\[
     s_{\lambda-\delta_k} (\bs{x}) s_{\lambda_u^{c,n-2}-\delta_m} (\bs{y})
    =
    \left( \sum\limits_S \bs{x}^S \right) \left( \sum\limits_T \bs{y}^T \right)
    =
    \sum\limits_{S,T} \bs{x}^S \bs{y}^T,
\]
where $S$ ranges over all tableaux of shape $\lambda-\delta_k$ and $T$ over all tableaux of shape $\lambda_u^{c,n-2}-\delta_m$.
\end{proof}

The case $u<0$ can be described by a variation of this construction.

\begin{corollary}
\label{cor:PartSchurTableauxnegative}
If $u<0$, then
\[
    S_u(\bs{x},\bs{y})
    =
    (-1)^{k(k-1)/2 + n(k-1)} \prod_{i=1}^k x_i^u \cdot
    \sum\limits_T (-1)^{\|\lambda\|} (\bs{x},\bs{y})^T,
\]
where the sum is over all tableaux $T$ formed as follows:
\begin{enumerate}
\item   Start with a Young diagram of shape
        $(n-2,n-3,\ldots,1,0)$, i.e. $n-1$ rows consisting of each integer from $0$ to $n-2$.
\item   Label $k-1$ of the rows $x$ and the remaining $m=n-k$ rows $y$.
        Note that $\lambda$ is the shape of the Young diagram labeled $x$ and $\lambda^{c,n-2}$
        is the shape of the Young diagram labeled $y$.
\item   Add $-u-1$ boxes to each row labeled $x$. Then, starting from the bottom row of length $0$, for each $j$,
        delete $j-1$ boxes from the $j$th row labeled $x$ and delete $j-1$ boxes from the
        $j$th row labeled $y$.
\item   Fill in the rows labeled $x$ with $x_1,\ldots,x_k$ so that the rows labeled $x$
        form a semistandard Young tableau of shape $\lambda-\delta_{k-1}-u-1$, and similarly fill in the rows
        labeled $y$ with $y_1,\ldots,y_m$ so that the rows labeled $y$ form a semistandard Young tableau of
        shape $\lambda^{c,n-2}-\delta_m$. As usual, $(\bs{x},\bs{y})^T$ denotes the product of the entries of $T$.
\end{enumerate}
\end{corollary}

\begin{proof} This is proven similarly to Corollary~\ref{cor:PartSchurTableaux}. Note in particular that
for $\lambda =(\lambda_1, \ldots,\lambda_n)$ and $\lambda-u =(\lambda_1-u, \ldots,\lambda_n-u)$,
we have a shifting rule for alternants
\[
    A_\lambda (\bs{x})
    =
    \begin{vmatrix}
        x_1^{\lambda_1} &   x_2^{\lambda_1} &   \cdots  &   x_n^{\lambda_1}
        \\
        x_1^{\lambda_2} &   x_2^{\lambda_2} &   \cdots  &   x_n^{\lambda_2}
        \\
        \vdots          &   \vdots          &           &   \vdots
        \\
        x_1^{\lambda_n} &   x_2^{\lambda_n} &   \cdots  &   x_n^{\lambda_n}
        \\
    \end{vmatrix}
    = \prod_{i=1}^n x_i^u \cdot
    \begin{vmatrix}
        x_1^{\lambda_1-u} &   x_2^{\lambda_1-u} &   \cdots  &   x_n^{\lambda_1-u}
        \\
        x_1^{\lambda_2-u} &   x_2^{\lambda_2-u} &   \cdots  &   x_n^{\lambda_2-u}
        \\
        \vdots          &   \vdots          &           &   \vdots
        \\
        x_1^{\lambda_n-u} &   x_2^{\lambda_n-u} &   \cdots  &   x_n^{\lambda_n-u}
        \\
    \end{vmatrix}
        = \prod_{i=1}^n x_i^u \cdot A_{\lambda-u} (\bs{x}).
\]
As a consequence,
\[
     s_{(\lambda,u)-\delta_k}(\bs{x}) = \prod_{i=1}^k x_i^u \cdot s_{(\lambda-u,0)-\delta_k}(\bs{x})
\]
allows us to interpret generalized Laurent-Schur polynomials in the context of classical Schur polynomials.
\end{proof}


\section{The coefficients of the Laurent series}
\label{sec:Laurent}

Let $\bs{a} = (a_1,\ldots,a_n)$ be the weight vector for an effective, stable action of $\C^\times$
on $\C^n$, where $a_i < 0$ for $i \leq k$ and $a_i > 0$ for $i > k$.
In this section, we detail a computation of the first few coefficients $\gamma_m(\bs{a})$ of
the Laurent series of $\Hilb_{\bs{a}}(t)$ at $t = 1$.

Using Equation~\eqref{eq:HilbSumDegen2Unif}, we consider the Laurent coefficients of the expression
\begin{equation}
\label{eq:Hac}
    H_{\bs{a},\bs{c}}(t) =
    \sum\limits_{i=1}^k \sum\limits_{\zeta^{-a_i} = 1}
        \frac{1}{-c_i \prod\limits_{\substack{j=1\\j\neq i}}^n 1 - \zeta^{a_j} t^{(c_i-c_j)/c_i}},
\end{equation}
where $\bs{c}=(c_1,\ldots,c_n)\in \R^n$ with $c_i < 0$ for $i \leq k$ and $c_i > 0$ for $i > k$.
As we will see, the resulting expressions for the $\gamma_m(H_{\bs{a},\bs{c}}(t))$ will be continuous
functions of the $c_i$ with poles occurring only at $c_i = c_j$ where $i \leq k$ and $j > k$ and hence not
in the domain under consideration. Noting that $H_{\bs{a},\bs{c}}(t)$ has a pole at $t = 1$
of order $n - 1$ (occurring in each term with $\zeta=1$), we can express
\[
    \gamma_m(H_{\bs{a},\bs{c}}(t))
    =
    \frac{1}{2\pi\sqrt{-1}}\int\limits_C H_{\bs{a},\bs{c}}(t) (t - 1)^{m-n} \, dt
\]
where $C$ is a positively oriented curve about $1$. Hence, as $H_{\bs{a},\bs{c}}(t)$
and the expressions we will derive for $\gamma_m(H_{\bs{a},\bs{c}}(t))$ are analytic
(up to removable singularities) in a punctured neighborhood of $t = 1$, an application of the
Dominated Convergence Theorem implies that
\[
    \lim_{\bs{c}\to\bs{a}} \gamma_m(H_{\bs{a},\bs{c}}(t)) = \gamma_m(\bs{a}).
\]
See \cite[Secion 5.2]{HerbigSeatonHilbSympCirc} for more details in the case of cotangent-lifted
representations; the argument applies without change to our setting.

\begin{remark}
\label{rem:kNotAbuse}
Our expressions for the $\gamma_m(\bs{a})$ will frequently involve the partial Schur
polynomials $S_u(\bs{a})$. We will always understand this notation to mean that
the weight vector $\bs{a}$ is split into two different sets of indeterminates, the
negative weights in the first set and the positive weights in the second. This slight
abuse of notation will be particularly convenient when we consider polynomials of the
form $S_u(\bs{a}_j)$, where $\bs{a}_j$ denotes the weight vector $\bs{a}$ with the $j$th
entry removed, as it will allow us to avoid using separate notation for the cases when
$j \leq k$ and $j > k$.
\end{remark}

\begin{theorem}
\label{thrm:Gamma0}
Let $\bs{a} = (a_1,\ldots,a_n)$ be the weight vector for an effective action of $\C^\times$ on $\C^n$.
Assume that $a_i < 0$ for $i \leq k$ and $a_i > 0$ for $i > k$, and moreover that
$\bs{a}$ is stable ($1 < k < n$). Then $\dim(\C[\C^n]^{\C_{\bs{a}}^\times}) = n - 1$, and
\begin{equation}
\label{eq:Gamma0}
    \gamma_0(\bs{a})
    =
    \frac{ - S_{n-2}(\bs{a})}
    {\prod\limits_{p=1}^k\prod\limits_{q=k+1}^n (a_p - a_q)}.
\end{equation}
In particular, $\gamma_0(\bs{a}) \neq 0$. When $\bs{a}$ is generic, this can be expressed as
\begin{equation}
\label{eq:Gamma0Generic}
    \gamma_0(\bs{a})
    =
    \sum\limits_{i=1}^k \frac{-a_i^{n-2} }
        {\prod\limits_{\substack{j=1\\j\neq i}}^n (a_i - a_j)}.
\end{equation}
\end{theorem}
\begin{proof}
Examining Equation~\eqref{eq:Hac}, we see that each term of $H_{\bs{a},\bs{c}}(t)$ has a pole order
of at most $n-1$, with this maximum obtained at the terms with $\zeta = 1$. Therefore,
$\gamma_0(H_{\bs{a},\bs{c}}(t))$ is the degree $1-n$ term in the Laurent series of
\begin{equation}
\label{eq:Zeta1Terms}
    \sum\limits_{i=1}^k
        \frac{1}{-c_i \prod\limits_{\substack{j=1\\j\neq i}}^n 1 - t^{(c_i-c_j)/c_i}}
    =
    \sum\limits_{i=1}^k \frac{1}{-c_i} \prod\limits_{\substack{j=1\\j\neq i}}^n
        \frac{1}{1 - t^{(c_i-c_j)/c_i}}.
\end{equation}
Using the fact that the Laurent series of $1/(1-t^c)$ at $t=1$ begins
\begin{equation}
\label{eq:LaurentGeneral}
    \frac{1}{1 - t^c}
    =
    \frac{1}{c}(1 - t)^{-1} + \frac{c - 1}{2c}
    + \frac{c^2-1}{12c}(1-t) + \frac{c^2-1}{24c}(1-t)^2 + \mathcal{O}\big((1-t)^3\big),
\end{equation}
we have that
\[
    \gamma_0(H_{\bs{a},\bs{c}}(t))
    =
    \sum\limits_{i=1}^k
        \frac{-c_i^{n-2} }
        {\prod\limits_{\substack{j=1\\j\neq i}}^n (c_i - c_j)},
\]
from which Equation~\eqref{eq:Gamma0Generic} follows. To express this as a single
rational function, we simplify
\begin{align}
    \label{eq:CombinRational}
    \gamma_0(H_{\bs{a},\bs{c}}(t))
    &=
    \sum\limits_{i=1}^k
        \frac{(-1)^i c_i^{n-2} }
        {\prod\limits_{j=1}^{i-1} (c_j - c_i)
        \prod\limits_{j=i+1}^n (c_i - c_j)}
    \\ \nonumber &=
    \frac{ \sum\limits_{i=1}^k (-1)^i c_i^{n-2}
        \prod\limits_{\substack{1 \leq p < q \leq n\\ p, q\neq i}} (c_p - c_q)}
        {\prod\limits_{1 \leq p < q \leq n} (c_p - c_q)},
\end{align}
where we recognize the numerator as a cofactor expansion of the form described in
Equation~\eqref{eq:PartSchurMotiv}. Therefore,
\[
    \gamma_0(H_{\bs{a},\bs{c}}(t))
    =
    \frac{ - S_{n-2}(\bs{c})}
    {\prod\limits_{p=1}^k\prod\limits_{q=k+1}^n (c_p - c_q)},
\]
where we note that the singularities at $c_i = c_j$ for $i,j\leq k$ or $i,j > k$ have
been removed. Taking the limit as $\bs{c}\to\bs{a}$ completes the proof
of Equation~\eqref{eq:Gamma0}.

By \cite[Remark 2]{WehlauPopov}, the assumption of stability implies that
$\dim (\C^n\git\C_{\bs{a}}^\times) = n - 1$. As the pole order of $\Hilb_{\bs{a}}(t)$
at $t = 1$ is equal to $\dim (\C^n\git\C_{\bs{a}}^\times)$ by \cite[Lemma 1.4.6]{DerskenKemperBook},
it follows that $\gamma_0(\bs{a}) \neq 0$.
\end{proof}

Before stating the next result, we introduce some additional notation.
For each $j$, let $\bs{a}_j \in \Z^{n-1}$ denote the weight vector $\bs{a}$ with $a_j$ removed
and let $g_j := \gcd\bs{a}_j = \gcd\{ a_i : i\neq j \}$.
For indeterminates $\bs{x} = (x_1,\ldots,x_n)$ and $1\leq j \leq n$, we let
\[
    E_j(\bs{x}) = \sum\limits_{1\leq i_1\leq\cdots\leq i_j\leq n} x_{i_1}\cdots x_{i_j}
\]
denote the elementary symmetric polynomial of degree $j$.

\begin{theorem}
\label{thrm:Gamma1}
Let $\bs{a} = (a_1,\ldots,a_n)$ be the weight vector for an effective action of $\C^\times$ on $\C^n$.
Assume that $a_i < 0$ for $i \leq k$ and $a_i > 0$ for $i > k$, and moreover that
$\bs{a}$ is stable ($1 < k < n$). Then
\begin{equation}
\label{eq:Gamma1}
    \gamma_1(\bs{a})
    =
    \frac{E_1(\bs{a}) S_{n-3}(\bs{a}) - S_{n-2}(\bs{a})}
        {2 \prod\limits_{p=1}^k\prod\limits_{q=k+1}^n (a_p - a_q) }
        +
        \sum\limits_{j=1}^n \left( \frac{g_j - 1}{2} \right) \gamma_0(\bs{a}_j).
\end{equation}
When $\bs{a}$ is generic, this can be expressed as
\begin{equation}
\label{eq:Gamma1Generic}
    \gamma_1(\bs{a})
    =
    \sum\limits_{i=1}^k \sum\limits_{\substack{j=1\\ j\neq i}}^n
        \frac{a_i^{n-3} a_j}
        {2\prod\limits_{\substack{\ell=1\\ \ell \neq i}}^n (a_i - a_\ell)}
        +
        \sum\limits_{i=1}^k \sum\limits_{\substack{j=1\\ j\neq i}}^n
        \left( \frac{g_j - 1}{2} \right)
        \frac{- a_i^{n-3}}{\prod\limits_{\substack{\ell=1 \\ \ell\neq i,j}}^n (a_i - a_\ell)}.
\end{equation}
\end{theorem}
\begin{proof}
Considering Equation~\eqref{eq:Hac}, a term of
$H_{\bs{a},\bs{c}}(t)$ contributes to the degree $2-n$ term of the Laurent series
if it has a pole order of either $n-1$ or $n-2$. The former case corresponds to
terms with $\zeta = 1$, while the latter corresponds to choices of $i$ and $a_i$th
root of unity $\zeta$ such that $\zeta^{a_j} = 1$ for all but one $j\neq i$.
We first consider the former.

Using the Laurent series in Equation~\eqref{eq:LaurentGeneral} and the Cauchy product
formula, the degree $2-n$ term of the Laurent series of the expression in
Equation~\eqref{eq:Zeta1Terms} for the terms with $\zeta=1$ is given by
\begin{equation}
\label{eq:Gamma1FirstTermsGeneric}
    \sum\limits_{i=1}^k \frac{1}{-c_i} \sum\limits_{\substack{j=1\\ j\neq i}}^n
    \left( \prod\limits_{\substack{\ell=1\\\ell\neq i,j}}^n \frac{ c_i }
    {(c_i - c_\ell)}
    \right)\frac{\big((c_i-c_j)/c_i - 1\big)c_i}{2(c_i - c_j)}
    =
    \sum\limits_{i=1}^k \sum\limits_{\substack{j=1\\ j\neq i}}^n
    \frac{c_i^{n-3} c_j}
    {2\prod\limits_{\substack{\ell=1\\ \ell \neq i}}^n (c_i - c_\ell)}.
\end{equation}
Combining via the same process as in Equation~\eqref{eq:CombinRational} above yields
\begin{align}
    \label{eq:Gamma1FirstRat}
    &=\frac{\sum\limits_{i=1}^k (-1)^{i-1} c_i^{n-3}
        \prod\limits_{\substack{1\leq p<q \leq n\\ p,q\neq i}} (c_p - c_q)
        \sum\limits_{\substack{j=1\\ j \neq i}}^n c_j}
        {2\prod\limits_{1\leq p<q \leq n} (c_p - c_q)}
    \\ \nonumber &=
        \frac{E_1(\bs{c}) \sum\limits_{i=1}^k (-1)^{i-1} c_i^{n-3}
        \prod\limits_{\substack{1\leq p<q \leq n\\ p,q\neq i}} (c_p - c_q)
        -
        \sum\limits_{i=1}^k (-1)^{i-1} c_i^{n-2}
        \prod\limits_{\substack{1\leq p<q \leq n\\ p,q\neq i}} (c_p - c_q)}
        {2\prod\limits_{1\leq p<q \leq n} (c_p - c_q)}
    \\ \label{eq:Gamma1FirstTerms}
    &=
    \frac{E_1(\bs{c}) S_{n-3}(\bs{c}) - S_{n-2}(\bs{c})}
        {2 \prod\limits_{p=1}^k\prod\limits_{q=k+1}^n (c_p - c_q) }.
\end{align}

We now consider the terms in $H_{\bs{a},\bs{c}}(t)$ with a pole order of $n-2$,
which correspond as noted above to choices of $i$ and $\zeta$
such that $\zeta^{a_j} = 1$ for all but one $j\neq i$. This clearly occurs only if there
is a $j$ such that $g_j \neq 1$ and $\zeta$ is a non-unit $g_j$th root of unity. Hence,
for such a choice of $i$ and $j$, the corresponding terms of $H_{\bs{a},\bs{c}}(t)$ are
\[
    \sum\limits_{\substack{\zeta^{g_j} = 1 \\ \zeta\neq 1}}
        \frac{1}{-c_i (1 - \zeta^{a_j} t^{(c_i-c_j)/c_i})
            \prod\limits_{\substack{\ell=1\\ \ell\neq i,j}}^n (1 - t^{(c_i-c_\ell)/c_i})}.
\]
Using Equation~\eqref{eq:LaurentGeneral} and noting that $\gcd(a_j, g_j) = \gcd(a_1,\ldots,a_n)=1$,
the first term of the Laurent series of this expression is
\begin{align}
    \nonumber
    \sum\limits_{\substack{\zeta^{g_j} = 1\\ \zeta\neq 1}}
            \frac{c_i^{n-2}}{-c_i(1 - \zeta^{a_j})
            \prod\limits_{\substack{\ell=1 \\ \ell\neq i,j}}^n (c_i - c_\ell)}
    &=      \frac{c_i^{n-2}}{-c_i \prod\limits_{\substack{\ell=1 \\ \ell\neq i,j}}^n (c_i - c_\ell)}
            \sum\limits_{\substack{\zeta^{g_j} = 1\\ \zeta\neq 1}} \frac{1}{1 - \zeta}
    \\ \label{eq:Gamma1LastTermsGeneric}
    &=    \frac{- c_i^{n-3}}{\prod\limits_{\substack{\ell=1 \\ \ell\neq i,j}}^n (c_i - c_\ell)}
            \left( \frac{g_j - 1}{2} \right),
\end{align}
where the sum over $\zeta$ is computed using \cite[Corollary 3.2]{Gessel}.

Along with taking the limit as $\bs{c}\to\bs{a}$ as described above,
combining Equations~\eqref{eq:Gamma1FirstTermsGeneric} and \eqref{eq:Gamma1LastTermsGeneric}
yields Equation~\eqref{eq:Gamma1Generic}. Then Equation~\eqref{eq:Gamma1} follows from
Equation~\eqref{eq:Gamma1FirstTerms} and comparing Equation~\eqref{eq:Gamma1LastTermsGeneric}
with Equation~\eqref{eq:Gamma0Generic} for $\gamma_0(\bs{a}_j)$.
\end{proof}

The approach used in Theorems~\ref{thrm:Gamma0} and~\ref{thrm:Gamma1} above can be used to compute
$\gamma_m(\bs{a})$ for $m > 1$, and the results can similarly be expressed in terms of the
$S_u(\bs{a})$. However, in these cases, one meets sums of rational expressions over roots of unity
as in Equation~\eqref{eq:Gamma1LastTermsGeneric} that are not readily computable using the results
of \cite{Gessel}. Recall \cite[(1.13)]{BeckRobinsBook} that a \emph{Fourier-Dedekind sum} is a sum of
the form
\[
    \sigma_r(a_2,\ldots,a_n;a_1)
    =
    \frac{1}{a_1} \sum\limits_{\substack{\zeta^{a_1}=1 \\ \zeta\neq 1}}
        \frac{\zeta^r }{\prod\limits_{j=2}^n (1 - \zeta^{a_j})}
\]
where each $a_2,\ldots,a_n$ is coprime to $a_1$; see also \cite{BeckDiazRobinsFDS,TsukermanFDS}.
The expressions below for $\gamma_m(\bs{a})$ with $m > 1$ involve nontrivial ``partial" Fourier--Dedekind
sums that bear a resemblance to Ramanujan's sum, i.e. sums over a subset of nonunit roots of unity, with
the coprime conditions on the $a_i$ relaxed.

For example, we have the following, whose proof involves tedious computations but requires
no more ingredients than those in the proof of Theorem~\ref{thrm:Gamma1}.
For each $j\neq \ell$, let $\bs{a}_{j,\ell} \in \Z^{n-2}$ denote the weight vector $\bs{a}$ with
$a_j$ and $a_\ell$ removed and let
$g_{j,\ell} := \gcd\bs{a}_{j,\ell} = \gcd\{ a_i : i\neq j,\ell \}$.

\begin{theorem}
\label{thrm:Gamma2}
Let $\bs{a} = (a_1,\ldots,a_n)$ be the weight vector for an effective action of $\C^\times$ on $\C^n$.
Assume that $a_i < 0$ for $i \leq k$ and $a_i > 0$ for $i > k$, and moreover that
$\bs{a}$ is stable ($1 < k < n$). Then
\begin{align*}
    \gamma_2(\bs{a}) &=
    \frac{5 E_1(\bs{a}) S_{n-3}(\bs{a}) - (E_2(\bs{a}) + E_1(\bs{a})^2)S_{n-4}(\bs{a})
        - 4 S_{n-2}(\bs{a})}
        {12\prod\limits_{p=1}^k\prod\limits_{q=k+1}^n (a_p - a_q)}
    \\&  \quad +
    \sum\limits_{j=1}^n \frac{1 - g_j^2}{12} \left(
        \frac{S_{n-3}(\bs{a}_j) - a_j S_{n-4}(\bs{a}_j)}
        {\prod\limits_{\substack{p=1\\ p\neq j}}^k
         \prod\limits_{\substack{q=k+1\\ q\neq j}}^n (a_p - a_q)}
        \right)
    \\& \quad +
    \sum\limits_{j=1}^n \frac{g_j-1}{4} \left(
        \frac{E_1(\bs{a}_j) S_{n-4}(\bs{a}_j)
        - S_{n-3}(\bs{a}_j)}
        {\prod\limits_{\substack{p=1\\ p\neq j}}^k\prod\limits_{\substack{q=k+1\\ q\neq j}}^n (a_p - a_q)}
        \right)
    \\&  \quad +
    \sum\limits_{1\leq j < \ell \leq n}
        \frac{ S_{n-4}(\bs{a}_{j,\ell})}
        {\prod\limits_{\substack{p=1\\ p\neq j,\ell}}^k\prod\limits_{\substack{q=k+1 \\ q\neq j,\ell}}^n (a_p - a_q)}
        \sum\limits_{\substack{\zeta^{g_{j,\ell}}=1 \\ \zeta^{g_j}\neq 1, \zeta^{g_\ell}\neq 1}}
        \frac{1}{(1 - \zeta^{a_j})(1 - \zeta^{a_\ell})}.
\end{align*}
When $\bs{a}$ is generic, this can be expressed as
\begin{align*}
    \gamma_2
    &=
    \sum\limits_{i=1}^k \frac{a_i^{n-4}}{12\prod\limits_{\substack{j=1 \\ j\neq i}}^n (a_i-a_j)}
    \left( \left( \sum\limits_{\substack{j=1\\ j\neq i}}^n
        (2a_i - a_j)a_j \right)
        - 3\sum\limits_{\substack{1\leq j < \ell \leq n \\ j,\ell\neq i}}^n a_j a_\ell \right)
    \\& \quad+
    \sum\limits_{i=1}^k\sum\limits_{\substack{j=1\\ j\neq i}}^n
        \left(
        \frac{(1 - g_j^2)a_i^{n-4} (a_i - a_j)}
        {12\prod\limits_{\substack{\ell=1\\ \ell\neq i,j}}^n (a_i-a_\ell)}
        + \frac{g_j - 1}{2}
        \sum\limits_{\substack{\ell=1\\ \ell\neq i,j}}^n \frac{a_i^{n-4} a_\ell}
        {2\prod\limits_{\substack{p=1\\ p\neq i,j}}^n (a_i-a_p)}\right)
    \\& \quad+
        \sum\limits_{i=1}^k \sum\limits_{\substack{1\leq j < \ell \leq n \\ j,\ell \neq i}}
        \frac{-a_i^{n-4}}{\prod\limits_{\substack{p=1\\p\neq i,j,\ell}}^n (a_i-a_p)}
        \sum\limits_{\substack{\zeta^{g_{j,\ell}}=1 \\ \zeta^{g_j}\neq 1, \zeta^{g_\ell}\neq 1}}
        \frac{1}{(1 - \zeta^{a_j})(1 - \zeta^{a_\ell})}.
\end{align*}
\end{theorem}

For the entertainment of the reader, we also state the following. Continuing the same notational convention as
above, for each distinct $j,\ell,p$, let $\bs{a}_{j,\ell,p} \in \Z^{n-3}$ denote the weight vector $\bs{a}$ with
$a_j$, $a_\ell$, and $a_p$ removed and let $g_{j,\ell,p} := \gcd\bs{a}_{j,\ell,p} = \gcd\{ a_i : i\neq j,\ell,p \}$.

\begin{theorem}
\label{thrm:Gamma3}
Let $\bs{a} = (a_1,\ldots,a_n)$ be the weight vector for an effective action of $\C^\times$ on $\C^n$.
Assume that $a_i < 0$ for $i \leq k$ and $a_i > 0$ for $i > k$, and moreover that
$\bs{a}$ is stable ($1 < k < n$). Then
\begin{align*}
    \gamma_3 &=
    \frac{-6 S_{n-2}(\bs{a})
        + (8 E_1(\bs{a}))S_{n-3}(\bs{a})
        - (3 E_2(\bs{a}) + 2 E_1(\bs{a})^2)S_{n-4}(\bs{a})
        + E_1(\bs{a}) E_2(\bs{a}) S_{n-5}(\bs{a})
    }{24\prod\limits_{p=1}^k\prod\limits_{q=k+1}^n (a_p - a_q)}
    \\  &\quad +
    \sum\limits_{j=1}^n \left(\frac{1 - g_j}{24} \right)
        \frac{4 S_{n-3}(\bs{a}_j)
        - (5 E_1(\bs{a}_j)) S_{n-4}(\bs{a}_j)
        + (E_2(\bs{a}_j) + E_1(\bs{a}_j)^2)S_{n-5}(\bs{a}_j)}
            {\prod\limits_{\substack{p=1\\ p\neq j}}^k\prod\limits_{\substack{q=k+1\\ q\neq j}}^n (a_p - a_q)}
\\
    &\quad +
    \sum\limits_{j=1}^n \frac{g_j^2-1}{24}
        \left(
        \frac{-2 S_{n-3}(\bs{a}_j)
        + (2a_j + E_1(\bs{a}_j))S_{n-4}(\bs{a}_j)
        - a_j E_1(\bs{a}_j) S_{n-5}(\bs{a}_j)}
            {\prod\limits_{\substack{p=1\\ p\neq j}}^k\prod\limits_{\substack{q=k+1\\ q\neq j}}^n (a_p - a_q)}
        \right)
    \\  &\quad +
    \sum\limits_{1\leq j < \ell \leq n}
        \sum\limits_{\substack{\zeta^{g_{j,\ell}}=1 \\ \zeta^{g_j}\neq 1, \zeta^{g_\ell}\neq 1}}
        \frac{1}{(1 - \zeta^{a_j})(1 - \zeta^{a_\ell})}
        \left(
        \frac{E_1(\bs{a}_{j,\ell})S_{n-5}(\bs{a}_{j,\ell})
            - S_{n-4}(\bs{a}_{j,\ell})}
            {2\prod\limits_{\substack{p=1\\ p\neq j,\ell}}^k \prod\limits_{\substack{q=k+1\\ q\neq j,\ell}}^n (a_p - a_q)}
        \right)
\\
    &\quad +
    \sum\limits_{1\leq j < \ell \leq n}
        \sum\limits_{\substack{\zeta^{g_{j,\ell}}=1 \\ \zeta^{g_j}\neq 1, \zeta^{g_\ell}\neq 1}}
        \left(
        \left( \frac{\zeta^{a_j}}{(1 - \zeta^{a_j})^2(1 - \zeta^{a_\ell})} \right)
            \frac{S_{n-4}(\bs{a}_{j,\ell})
                - a_j S_{n-5}(\bs{a}_{j,\ell})}
            {\prod\limits_{\substack{p=1\\ p\neq j,\ell}}^k
            \prod\limits_{\substack{q=k+1\\ q\neq j,\ell}}^n (a_p - a_q)}
        \right.
    \\ &\quad\quad\quad +
        \left.
        \left( \frac{\zeta^{a_\ell}}{(1 - \zeta^{a_j})(1 - \zeta^{a_\ell})^2} \right)
            \frac{S_{n-4}(\bs{a}_{j,\ell})
                - a_\ell S_{n-5}(\bs{a}_{j,\ell})}
            {\prod\limits_{\substack{p=1\\ p\neq j,\ell}}^k
            \prod\limits_{\substack{q=k+1\\ q\neq j,\ell}}^n (a_p - a_q)}
        \right)
    \\ &\quad +
    \sum\limits_{1\leq j < \ell < p \leq n}
    \frac{ -S_{n-5}(\bs{a}_{j,\ell,p}) }
    {\prod\limits_{\substack{p=1 \\ p\neq j,\ell, p}}^k
    \prod\limits_{\substack{q=k+1 \\ q \neq j, \ell, p}}^n (a_p - a_q)}
    \sum\limits_{\substack{\zeta^{g_{j,\ell,p}}=1 \\ \zeta^{g_{j,\ell}}\neq 1, \zeta^{g_{j,p}}\neq 1 \\
        \zeta^{g_{\ell,p}}\neq 1}}
        \frac{1}{(1 - \zeta^{a_j})(1 - \zeta^{a_\ell})(1 - \zeta^{a_p})}.
\end{align*}
When $\bs{a}$ is generic,
\begin{align*}
    \gamma_3
    &=
    \sum\limits_{i=1}^k
    \frac{a_i^{n-5} \left(
        \sum\limits_{\substack{1\leq j < \ell < p \leq n \\ j,\ell,p\neq i}} 3a_j a_\ell a_p
        +
        \sum\limits_{\substack{j=1\\ j\neq i}}^n \sum\limits_{\substack{\ell=1\\ \ell\neq i,j}}^n
        a_j a_\ell(a_\ell - 2a_i)
        +
        \sum\limits_{\substack{j=1\\ j\neq i}}^n a_i a_j(2a_i - a_j)
        \right)
        }{24\prod\limits_{\substack{q=1 \\ q\neq i}}^n (a_i-a_q)}
\\&\quad+
    \sum\limits_{i=1}^k \sum\limits_{\substack{j=1\\ j\neq i}}^n \left(
    \frac{g_j-1}{2} \left( \sum\limits_{\substack{1\leq\ell < p \leq n\\ \ell,p\neq i,j}}^n
        \frac{-a_i^{n-5} a_\ell a_p}{4\prod\limits_{\substack{q=1\\q\neq i,j}}^n(a_i-a_q)}
        +
        \sum\limits_{\substack{\ell=1 \\ \ell\neq i,j}}^n
            \frac{-a_i^{n-5} a_\ell(a_\ell -2a_i)}
            {12\prod\limits_{\substack{p=1\\p\neq i,j}}^n(a_i-a_p)}
        \right)\right.
    \\  &\quad
        \left.
        + \frac{g_j^2-1}{24}\left(
        \frac{a_i^{n-5}(a_i-a_j)\left(-a_i + \sum\limits_{\substack{\ell=1 \\ \ell\neq i,j}}^n a_\ell \right)}
        {\prod\limits_{\substack{\ell=1\\\ell\neq i,j}}^n (a_i-a_\ell)}
        \right)\right)
    \\ &+
    \sum\limits_{i=1}^k \sum\limits_{\substack{1\leq j < \ell \leq n\\j,\ell\neq i}}
    \left(
    \sum\limits_{\substack{\zeta^{g_{j,\ell}}=1 \\ \zeta^{g_j}\neq 1, \zeta^{g_\ell}\neq 1}}
        \frac{1}{(1 - \zeta^{a_j})(1 - \zeta^{a_\ell})}
        \sum\limits_{\substack{p=1\\p\neq i,j,\ell}}^n
        \frac{a_i^{n-5} a_p}{2\prod\limits_{\substack{q=1\\q\neq i,j,\ell}}^n (a_i - a_q)}
    \right.
\\
    &\quad
    \left.
    + \frac{a_i^{n-5}}{\prod\limits_{\substack{p=1\\p\neq i,j,\ell}}^n (a_i - a_p)}
        \sum\limits_{\substack{\zeta^{g_{j,\ell}}=1 \\ \zeta^{g_j}\neq 1, \zeta^{g_\ell}\neq 1}}
        \left( \frac{\zeta^{a_j}(a_i - a_j)}{(1 - \zeta^{a_j})^2(1 - \zeta^{a_\ell})}
        + \frac{\zeta^{a_\ell}(a_i - a_\ell)}{(1 - \zeta^{a_j})(1 - \zeta^{a_\ell})^2}\right)\right)
    \\ &+
    \sum\limits_{i=1}^k \sum\limits_{\substack{1\leq j < \ell < p \leq n\\j,\ell,p\neq i}}
    \sum\limits_{\substack{\zeta^{g_{j,\ell,p}}=1 \\ \zeta^{g_{j,\ell}}\neq 1, \zeta^{g_{j,p}}\neq 1 \\
        \zeta^{g_{\ell,p}}\neq 1}}\frac{-a_i^{n-5}}
        {(1 - \zeta^{a_j})(1 - \zeta^{a_\ell})(1 - \zeta^{a_p})\prod\limits_{\substack{q=1\\q\neq i,j,\ell,p}}^n (a_i - a_q)}.
\end{align*}
\end{theorem}

As complicated as the expressions in Theorems~\ref{thrm:Gamma2} and \ref{thrm:Gamma3} appear,
it is important to note that they become vastly simpler when one imposes mild coprime
hypotheses on the weights. For instance, if we assume that no set of $n-2$ weights has a
nontrivial common divisor, then each $g_j = g_{j,\ell} = 1$. Hence in both expressions
for $\gamma_2$, all but the first line vanishes. In particular, the generalized Fourier-Dedekind
sum no longer appears. The same holds for $\gamma_3$ if we assume that no set of $n-3$ weights
has a common divisor.

To indicate one application of these formulas, note that the $\gamma_m(A)$ can be expressed
in terms of the degrees of the elements of a Hironaka decomposition of $A$; see
Appendix~\ref{app:CMAlg}. Hence, Theorems~\ref{thrm:Gamma0}, \ref{thrm:Gamma1}, \ref{thrm:Gamma2},
and \ref{thrm:Gamma3} can be used to determine bounds on these degrees. To illustrate this in
a simple case, assume $\bs{a}$ is a generic weight vector with $k = 1$. Examining
Equation~\eqref{eq:Gamma0Generic}, it is easy to see that $\gamma_0(\bs{a}) \leq 1/2$.
From Equation~\eqref{eq:gammaexpressions}, it then follows that the number of module
generators in the Hironaka decomposition, so-called \emph{secondary invariants}, is
bounded by half the product of the degrees of elements of a homogeneous system of parameters,
the \emph{primary invariants}. While we will not pursue this line of reasoning further here,
it would be interesting to conduct numerical experiments to identify if such comparisons yield
strong restrictions on the number and degrees of primary and secondary invariants.

We end this section with illustrations of Theorems~\ref{thrm:Gamma0} and~\ref{thrm:Gamma1} for
small values of $n$.

\begin{example}
\label{ex:GammasN2}
Suppose $n=2$. Then stability implies $k = 1$ so that $\bs{a} = (a_1,a_2)$ with
$a_1 < 0$ and $a_2 > 0$. Then Theorems~\ref{thrm:Gamma0}, \ref{thrm:Gamma1},
\ref{thrm:Gamma2}, and \ref{thrm:Gamma3} yield
\begin{align*}
    \gamma_0(\bs{a})
        &=  \frac{-1}{a_1 - a_2},
    &\gamma_1(\bs{a})
        &=  \frac{1 + a_1 - a_2}{2(a_1 - a_2)},
    &\gamma_2(\bs{a})
        &=  \frac{1 - (a_1 - a_2)^2}{12(a_1 - a_2)},
        \quad\quad\mbox{and}
    &\gamma_3(\bs{a})
        &=  \frac{1 - (a_1 - a_2)^2}{24(a_1 - a_2)}.
\end{align*}
\end{example}

\begin{example}
\label{ex:GammasN3}
If $n = 3$, multiplying by $-1$ if necessary, we may assume $k = 1$. Then
$\bs{a} = (a_1,a_2,a_3)$ with $a_1 < 0$ and $a_2, a_3 > 0$. The first two
Laurent coefficients are in this case given by
\begin{align*}
    \gamma_0(\bs{a})
        &=  \frac{-a_1}{(a_1 - a_2)(a_1 - a_3)},
            \quad\quad\mbox{and}
    &\gamma_1(\bs{a})
        &=  \frac{2a_1 + (a_3 - a_1)\gcd(a_1,a_2) + (a_2 - a_1)\gcd(a_1,a_3)}
                {2(a_1 - a_2)(a_1 - a_3)}.
\end{align*}
Of course, if $a_1$ is assumed coprime to both $a_2$ and $a_3$, this simplifies to
$\gamma_1(\bs{a}) = (a_2 + a_3)/(2(a_1 - a_2)(a_1 - a_3))$.
In this case, $g_{1,2} = a_3$, so unless each weight is $\pm 1$, $\gamma_2(\bs{a})$ always involves
sums of the form $\sum 1/((1-\zeta^{a_i})(1-\zeta^{a_j}))$.
\end{example}

\begin{example}
\label{ex:GammasN4}
If $n = 4$, then up to multiplying by $-1$, $k$ is either $1$ or $2$.
If $k = 1$, then
\begin{align*}
    \gamma_0(\bs{a})
        &=  \frac{-a_1^2}{(a_1 - a_2)(a_1 - a_3)(a_1 - a_4)},
            \quad\quad\mbox{and}
    \\
    \gamma_1(\bs{a})
        &=  \frac{a_1 \big(3a_1 - (a_1 - a_4)\gcd(a_1,a_2,a_3)
                - (a_1 - a_3)\gcd(a_1,a_2,a_4) - (a_1 - a_2)\gcd(a_1,a_3,a_4)\big)}
            {2(a_1 - a_2)(a_1 - a_3)(a_1 - a_4)}.
\end{align*}
If $k = 2$, then
\begin{align*}
    \gamma_0(\bs{a})
        &=  \frac{a_1 a_2 (a_3 + a_4) - (a_1 + a_2) a_3 a_4}
                {(a_1 - a_3)(a_1 - a_4)(a_2 - a_3)(a_2 - a_4)},
            \quad\quad\mbox{and}
    \\
    \gamma_1(\bs{a})
        &=  \frac{3\big((a_1 + a_2)a_3 a_4 - a_1 a_2 (a_3 + a_4)\big)}
                {2(a_1 - a_3)(a_1 - a_4)(a_2 - a_3)(a_2 - a_4)}
        \\&\quad + \frac{a_3 (a_1 - a_4) (a_2 - a_4) \gcd(a_1, a_2, a_3)
                        + (a_1 - a_3) (a_2 - a_3) a_4 \gcd(a_1, a_2, a_4)}
                    {2(a_1 - a_3)(a_1 - a_4)(a_2 - a_3)(a_2 - a_4)}
        \\&\quad - \frac{a_1 (a_2 - a_3)(a_2 - a_4) \gcd(a_1, a_3, a_4)
                        + a_2 (a_1 - a_3)(a_1 - a_4) \gcd(a_2, a_3, a_4)}
                    {2(a_1 - a_3)(a_1 - a_4)(a_2 - a_3)(a_2 - a_4)}.
\end{align*}
\end{example}


\section{The $a$-invariant and the Gorenstein property}
\label{sec:Goren}

Recall from Section~\ref{sec:Back} that the $a$-invariant of $\C[\C^n]^{\C_{\bs{a}}^\times}$ is
given by the degree of $\Hilb_{\bs{a}}(t)$, i.e. the degree of the numerator minus the degree
of the denominator (also equal to the pole order at infinity). Hence, the $a$-invariant can
be computed for specific $\bs{a}$ using the algorithm described in Section~\ref{sec:DirtyMethod}.
However, if $\C[\C^n]^{\C_{\bs{a}}^\times}$ is Gorenstein, then we may use Equation~\eqref{eq:AInvGammas}
and Theorems~\ref{thrm:Gamma0} and \ref{thrm:Gamma1} to give an explicit formula for
$a(\C[\C^n]^{\C_{\bs{a}}^\times})$ in terms of the weights as follows.

\begin{corollary}
\label{cor:a-invar}
Let $\bs{a} = (a_1,\ldots,a_n)$ be the weight vector for an effective action of $\C^\times$ on $\C^n$.
Assume that $a_i < 0$ for $i \leq k$ and $a_i > 0$ for $i > k$, and moreover that
$\bs{a}$ is stable ($1 < k < n$). If the invariant ring $\C[\C^n]^{\C_{\bs{a}}^\times}$
is Gorenstein, then the $a$-invariant is given by
\begin{align}
\label{eq:AInvFull}
    a(\C[\C^n]^{\C_{\bs{a}}^\times})
    &=
    \frac{1}{S_{n-2}(\bs{a})}\left(
        E_1(\bs{a}) S_{n-3}(\bs{a})
        + \sum\limits_{j=1}^k
        (1 - g_j)S_{n-2}(\bs{a}_j) \prod\limits_{q=k+1}^n (a_j - a_q)
        \right.
    \\ \nonumber & \left. \quad\quad\quad\quad\quad\quad\quad\quad\quad\quad +
        \sum\limits_{j=k+1}^n
        (1 - g_j)S_{n-2}(\bs{a}_j) \prod\limits_{p=1}^k (a_p - a_j) \right)
         - n,
\end{align}
where $g_j$ and $\bs{a}_j$ are defined as in Section~\ref{sec:Laurent}.
In particular, if every collection of $n-1$ weights has no nontrivial common factor, then
\begin{equation}
\label{eq:AInvRelPrime}
    a(\C[\C^n]^{\C_{\bs{a}}^\times})
    =
    \frac{E_1(\bs{a}) S_{n-3}(\bs{a})}{S_{n-2}(\bs{a})} - n.
\end{equation}
When $\bs{a}$ is generic, we can express Equation~\eqref{eq:AInvFull} as
\begin{equation}
\label{eq:AInvGeneric}
    a(\C[\C^n]^{\C_{\bs{a}}^\times})
    =
    \frac{\sum\limits_{i=1}^k (-1)^{i-1} a_i^{n-3}
        \prod\limits_{\substack{1\leq p<q \leq n\\ p,q\neq i}} (a_p - a_q)
        \left(\sum\limits_{\substack{j=1\\ j \neq i}}^n a_j
            - (g_j - 1)(c_i - c_j) \right)}
        {\sum\limits_{i=1}^k (-1)^{i} a_i^{n-2}
            \prod\limits_{\substack{1\leq p<q \leq n\\ p,q\neq i}} (a_p - a_q)}.
\end{equation}
\end{corollary}

The proof of Equation~\eqref{eq:AInvFull} is a simple computation using Equations~\eqref{eq:Gamma0}
and \eqref{eq:Gamma1}; Equation~\eqref{eq:AInvGeneric} is similarly derived using
Equations~\eqref{eq:Gamma0Generic}, \eqref{eq:Gamma1FirstRat}, and \eqref{eq:Gamma1LastTermsGeneric}.
The fact that $S_{n-2}(\bs{a}) \neq 0$ so that this expression is defined is a consequence of the fact
that $\gamma_0 \neq 0$, see Theorem~\ref{thrm:Gamma0}.

\begin{example}
\label{ex:AInvN2}
Suppose $n=2$ so that $\bs{a} = (a_1,a_2)$ with
$a_1 < 0$ and $a_2 > 0$. Then using Example~\ref{ex:GammasN2},
\[
    \frac{2\gamma_1(\bs{a})}{\gamma_0(\bs{a})}
    =
    a_2 - a_1 - 1
\]
so that if $\C[\C^n]^{\C_{\bs{a}}^\times}$ is Gorenstein,
\[
    a(\C[\C^n]^{\C_{\bs{a}}^\times})
    =
    -\frac{2\gamma_1(\bs{a})}{\gamma_0(\bs{a})} - \dim(\C[\C^n]^{\C_{\bs{a}}^\times})
    =
    a_1 - a_2.
\]
\end{example}

\begin{example}
\label{ex:AInvN3}
If $n=3$ and $k=1$ so that $\bs{a} = (a_1,a_2,a_3)$ with
$a_1 < 0$ and $a_2,a_3 > 0$, then using Example~\ref{ex:GammasN3}, we have
\[
    \frac{2\gamma_1(\bs{a})}{\gamma_0(\bs{a})}
    =
    - \frac{2a_1 + (a_3 - a_1)\gcd(a_1,a_2) + (a_2 - a_1)\gcd(a_1,a_3)}
                {a_1}
\]
so that if $\C[\C^n]^{\C_{\bs{a}}^\times}$ is Gorenstein
\[
    a(\C[\C^n]^{\C_{\bs{a}}^\times})
    =
    -\frac{2\gamma_1(\bs{a})}{\gamma_0(\bs{a})} - \dim(\C[\C^n]^{\C_{\bs{a}}^\times})
    =
    \frac{(a_3 - a_1)\gcd(a_1,a_2) + (a_2 - a_1)\gcd(a_1,a_3)}
                {a_1}.
\]
\end{example}

Of course, as the $a$-invariant $a(\C[\C^n]^{\C_{\bs{a}}^\times})$ and dimension
$\dim(\C[\C^n]^{\C_{\bs{a}}^\times})$ are always integers, an immediate consequence of
Equation~\eqref{eq:AInvGammas} is the following.

\begin{corollary}
\label{cor:GorensteinInteger}
Let $\bs{a} = (a_1,\ldots,a_n)$ be the weight vector for an effective action of $\C^\times$ on $\C^n$.
Assume that $a_i < 0$ for $i \leq k$ and $a_i > 0$ for $i > k$, and moreover that
$\bs{a}$ is stable ($1 < k < n$). If the invariant ring $\C[\C^n]^{\C_{\bs{a}}^\times}$ is Gorenstein, then
$2\gamma_1(\bs{a})/\gamma_0(\bs{a}) \in \Z$.
\end{corollary}

The converse of Corollary~\ref{cor:GorensteinInteger} is false, which we illustrate with the following.

\begin{example}
\label{ex:AInvIntNotGoren}
Let $\bs{a} = \{-1,-2,1,14\}$. Using the algorithm described in Section~\ref{sec:DirtyMethod}, one
computes that
\[
    \Hilb_{\bs{a}}(t) = \frac{1 + t^{3} + t^{6} + 2t^{9} + t^{10} + t^{11} + 2t^{12} + t^{13} + t^{14} + t^{15}}
                        {(1 - t^2)(1 - t^8)(1 - t^{15})},
\]
with $\gamma_0(\bs{a}) = 1/20$ and $\gamma_1(\bs{a}) = 3/40$ so that
$2\gamma_1(\bs{a})/\gamma_0(\bs{a}) = 3$. However, if $\C[\C^n]^{\C_{\bs{a}}^\times}$ were Gorenstein,
then Equation~\eqref{eq:AInvGammas} would imply that $a(\C[\C^n]^{\C_{\bs{a}}^\times}) = 6$, which does
not coincide with the degree of $\Hilb_{\bs{a}}(t)$. Moreover, by inspection, $\Hilb_{\bs{a}}(t)$ does not
satisfy Equation~\eqref{eq:StanleyGoren} for any integer $a$.
\end{example}

Though we have not identified a counterexample to the converse of Corollary~\ref{cor:GorensteinInteger}
with $n = 3$, they seem to be common when $n = 4$ and $k = 2$; other examples include
$(-1, -2, 4, 8)$;  $(-1, -2, 5, 6)$; $(-1, -3, 1, 27)$; $(-1, -3, 2, 9)$; $(-1, -3, 3, 9)$;
$(-1, -3, 4, 6)$; $(-1, -3, 12, 23)$; and $(-1, -4, 2, 2)$.

For a specific $\bs{a}$ with small weights, the most direct way of determining whether
$\C[\C^n]^{\C_{\bs{a}}^\times}$ is Gorenstein using the results in this paper is to
compute $\Hilb_{\bs{a}}(t)$ using the algorithm
of Section~\ref{sec:DirtyMethod} and testing to see if it satisfies Equation~\eqref{eq:StanleyGoren}.
However, as the weights become large, Corollary~\ref{cor:GorensteinInteger} can be a surprisingly useful way
of quickly determining that the invariant ring associated to a weight vector $\bs{a}$ is not
Gorenstein. For instance, for the weight vector
$\bs{a} = (-501, 500, 503)$, a quick computation by hand using Corollary~\ref{cor:GorensteinInteger}
demonstrates that $\C[\C^n]^{\C_{\bs{a}}^\times}$ is not Gorenstein, as
$2\gamma_1(\bs{a})/\gamma_0(\bs{a}) = -1003/501 \notin\Z$. However, our implementation of the algorithm
in Section~\ref{sec:DirtyMethod} took over six hours on a desktop PC to compute $\Hilb_{(-501,500,503)}(t)$
and yield the same conclusion.

We conclude this section with some illustrations of how Theorem~\ref{thrm:HilbSer} can be used to identify
classes of weight vectors with Gorenstein invariants. We first give a quick proof of the known fact that
when $n=2$, the invariant ring $\C[\C^n]^{\C_{\bs{a}}^\times}$ is always polynomial; see \cite{WehlauPolynomial}.

\begin{corollary}
\label{cor:GorenN2}
Suppose $\bs{a} = (a_1,a_2) \in \Z^2$ with $a_1 < 0$ and $a_2 > 0$. Then $\C[\C^n]^{\C_{\bs{a}}^\times}$
is a polynomial ring generated by $x_1^{a_2} x_2^{a_1}$. In particular, $\C[\C^n]^{\C_{\bs{a}}^\times}$
is Gorenstein with $a$-invariant $a_1 - a_2$.
\end{corollary}
\begin{proof}
Assume for simplicity that $\gcd(a_1,a_2) = 1$, and then by Theorem~\ref{thrm:HilbSer},
\begin{align*}
    \Hilb_{\bs{a}}(t)
    &=
    \sum\limits_{\zeta^{-a_1} = 1}
        \frac{1}{-a_1 (1 - \zeta^{a_2} t^{(a_1-a_2)/a_1})}
    \\&=
    \frac{-1}{a_1}\sum\limits_{\zeta^{-a_1} = 1}
        \frac{1}{1 - \zeta t^{(a_1-a_2)/a_1}},
\end{align*}
where the second equation is by reordering terms, as $\zeta^{a_2}$ simply permutes the set of $a_1$st
roots of unity. By \cite[Theorem 3.1]{Gessel}, this is equal to $1/(1 - t^{a_2 - a_1})$. Then as
$x_1^{a_2} x_2^{a_1}$ is clearly $\C^\times_{\bs{a}}$-invariant, it then generates $\C[\C^n]^{\C_{\bs{a}}^\times}$,
and the result follows.
\end{proof}

Finally, for general $n$, we indicate a large class of $\bs{a}$ for which $\C[\C^n]^{\C_{\bs{a}}^\times}$
is Gorenstein.

\begin{corollary}
\label{cor:GorenK1}
Let $\bs{a}\in \Z^n$ with each $a_i \neq 0$ and suppose $k=1$ so that $a_1 < 0$ and $a_i > 0$ for $i > 1$.
If $a_1$ divides $\sum_{j=2}^n a_j$, then $\C[\C^n]^{\C_{\bs{a}}^\times}$ is Gorenstein
with $a$-invariant $(\sum_{j=1}^n a_j)/a_1 - n$.
\end{corollary}
\begin{proof}
If $k=1$, then the weight vector is always generic, and Theorem~\ref{thrm:HilbSer} yields
\[
    \Hilb_{\bs{a}}(t)
    =
    \sum\limits_{\zeta^{-a_1} = 1}
        \frac{1}{-a_1 \prod\limits_{j=2}^n (1 - \zeta^{a_j} t^{(a_1-a_j)/a_1})}.
\]
Applying Equation~\eqref{eq:StanleyGoren}, we have
\begin{align*}
    \Hilb_{\bs{a}}(1/t)
    &=
    \sum\limits_{\zeta^{-a_1} = 1}
        \frac{1}{-a_1 \prod\limits_{j=2}^n (1 - \zeta^{a_j} t^{-(a_1-a_j)/a_1})}
    \\&=
    t^{n - 1 - \sum_{j=2}^n a_j/a_1}
    \sum\limits_{\zeta^{-a_1} = 1}
        \frac{\zeta^{-\sum_{j=2}^n a_j}}
        {-a_1 \prod\limits_{j=2}^n (\zeta^{-a_j} t^{(a_1-a_j)/a_1} - 1)}.
\end{align*}
Reordering terms by replacing $\zeta$ with $\zeta^{-1}$ yields
\[
    \Hilb_{\bs{a}}(1/t)
    =   (-1)^{n-1} t^{n - 1 - (\sum_{j=2}^n a_j)/a_1}
        \sum\limits_{\zeta^{-a_1} = 1}
        \frac{\zeta^{\sum_{j=2}^n a_j}}
        {-a_1 \prod\limits_{j=2}^n (1 - \zeta^{a_j} t^{(a_1-a_j)/a_1})},
\]
which, if $a_1$ divides $\sum_{j=2}^n a_j$, is equal to
\[
    (-1)^{n-1} t^{n - (\sum_{j=1}^n a_j)/a_1} \Hilb_{\bs{a}}(t).
    \qedhere
\]
\end{proof}

Note that the condition of Corollary~\ref{cor:GorenK1} is sufficient though not necessary:
the ring of invariants corresponding to the weight vector $\bs{a} = (-3,1,3)$ has Hilbert series
\[
    \frac{1}{(1 - t^2)(1 - t^4)}
\]
and hence is Gorenstein with $a$-invariant $-6$.

\appendix

\section{The Laurent expansion of the Hilbert series of a Cohen-Macaulay algebra}
\label{app:CMAlg}

Let $A = \oplus_{i=0}^\infty A_i$ be a Cohen-Macaulay algebra over the field $\mathbb{K}$, let
$d = \dim(A)$, and let $x_1,x_2,\dots,x_d$ be a homogeneous system of parameters with respective
degrees $\alpha_1,\alpha_2,\dots, \alpha_d$. Furthermore, let $\beta_1,\beta_2,\dots,\beta_r$ be the
degrees of a basis of the free $\mathbb{K}[x_1,x_2,\dots,x_d]$-module $A$. Such a choice of homogeneous
system of parameters and module generators is called a \emph{Hironaka decomposition}. The Hilbert series
of $A$ can be written as
\[
    \Hilb_A(t)
        =   \sum\limits_{m=0}^\infty \dim_{\mathbb{K}}(A_m)\:t^m
        =   \frac{\sum_{i=1}^r t^{\beta_i}}{\prod_{j=0}^d (1 - t^{\alpha_j})},
\]
see \cite[Equation 3.28]{PopovVinberg} or \cite[Corollary 2.3.4]{SturmfelsBook}. Our aim in this section is to
elaborate a formula for the Laurent expansion
\[
    \Hilb_A(t)
        =   \sum\limits_{m=0}^\infty \gamma_m(A)(1 - t)^{m - d}
\]
using symmetric functions in the $\alpha_1,\alpha_2,\ldots,\alpha_d$ and $\beta_1,\beta_2,\ldots,\beta_r$, respectively.
For the $\alpha_i$'s we use the elementary symmetric functions, which we denote
\[
    e_k     :=  E_k(\alpha_1,\alpha_2,\dots, \alpha_d)
            =   \sum\limits_{0\le j_1<j_2<\cdots <j_k\le d}
                    \alpha_{j_1}\alpha_{j_2} \cdots \alpha_{j_d}, \quad k\ge 0,
\]
while for the $\beta$'s, we use the power sums
\[
    p_k     :=  P_k(\beta_1,\beta_2,\dots, \beta_r)
            =   \sum\limits_{i=1}^r \beta_i^k, \quad 0 \leq k \leq r.
\]

It is well-known \cite[Equations 3.29 and 3.30]{PopovVinberg} that
\begin{equation}
\label{eq:gammaexpressions}
    \gamma_0(A) =   \frac{r}{e_d},
    \quad\quad\mbox{and}\quad\quad
    \gamma_1(A) =   \frac{\frac{r}{2}(e_1 - d) - p_1}{e_d}.
\end{equation}
In order to deduce a generalization of these formulas, we recall the definition of the Todd
polynomials $\operatorname{td}_j$ using the generating function
\begin{align*}
    \prod\limits_{i=1}^d \frac{x \alpha_i}{1 - e^{-x \alpha_i}}
        &=  :\sum\limits_{j=0}^\infty \operatorname{td}_j(e_1,e_2,\dots,e_d) \:x^j         \\
        &=  1 + \frac{e_1}{2}x + \frac{e_1^2 + e_2}{12}x^2 + \frac{e_1e_2}{24}x^3
                + \frac{-e_1^4 + 4e_1^2e_2 + e_1e_3 + 3e_2^2 - e_4}{720}x^4 + \cdots ,
\end{align*}
see \cite[Section 1.7]{Hirzebruch}. Substituting $x = -\log t$ we find
\[
    e_d(-\log t)^d \prod_{i=1}^d \frac{1}{1 - t^{\alpha_i}}
        =   \sum\limits_{j=0}^\infty \operatorname{td}_j(e_1,e_2,\dots,e_d)\; (-\log t)^j.
\]
For $k\in \Z$, we define $\lambda_m$ by the expansion
\[
    (-\log t)^k     =:      \sum\limits_{m=0}^\infty \lambda_m(k)(1-t)^{m+k}
\]
and observe that $\lambda_m(k)$ is a polynomial in $k$ of degree $m$. The first few $\lambda_m(k)$ are
\begin{align*}
    \lambda_0(k)    =   1, \quad
    \lambda_1(k)    =   \frac{k}{2}, & \quad
    \lambda_2(k)    =   \frac{k}{24} (3k+5),             \\
    \lambda_3(k)    =   \frac{k}{48} (k^2+5k+6), \quad
    \lambda_4(k)   &=   \frac{k}{5760} (15k^3 + 150k^2 + 485k + 502).
\end{align*}
Note that the $\lambda_m(k)$ are determined implicitly by the recursion
\begin{align*}
    (m+k)\lambda_m(k) - k\lambda_m(k-1)
        &=      (m + k - 1)\lambda_{m-1}(k),
    \\
    \lambda_0(k) = 1,
    \quad\quad\quad
    \lambda_m(0)
    &= 0
    \quad\mbox{for}\quad m > 1.
\end{align*}

Defining
\[
    \varphi_m   :=  \varphi_m(e_1,e_2,\dots,e_d)
                =   \sum_{k=0}^m \lambda_{m-k}(k-d) \operatorname{td}_k(e_1,e_2,\dots,e_d),
\]
we find
\[
    \prod\limits_{i=1}^d \frac{1}{1 - t^{\alpha_i}}
        =   \frac{1}{e_d} \sum\limits_{m=0}^\infty \frac {\varphi_m(e_1,e_2,\dots,e_d)}{(1 - t)^{d - m}}.
\]
The first few $\varphi_m$ are
\begin{align*}
    \varphi_0   &=  1 , \quad
    \varphi_1   =   \frac{1}{2} (e_1 - d), \quad
    \varphi_2   =   \frac{1}{12} \left(e_2 + e_1^2 - 3(d-1)e_1 + \frac{d}{2}(3d - 5)\right),
    \\
    \varphi_3   &=  \frac{1}{24}\left(e_2e_1 - (d - 2)(e_2 + e_1^2) + \frac{d - 1}{2}(3d - 8)e_1
                        - \frac{d(d - 2)(d - 3)}{2} \right).
\end{align*}

In order to incorporate the numerator $\sum_i t^{\beta_i}$, we recall the definition of the Stirling numbers
of the first kind $s(m,k)$ in terms of falling factorials
\[
    x(x - 1)(x - 2)\cdots (x - m + 1) =   \sum_{k=0}^m s(m,k) x^k,
\]
see \cite[Section 24.1.3]{AbramowitzStegun}.
Using
\[
    \sum_{i=1}^r t^{\beta_i}
        =   \sum\limits_{i=1}^r \big(1 - (1 - t)\big)^{\beta_i}
        =   \sum\limits_{i=1}^r \sum\limits_{j=0}^{\beta_i} (-1)^j \binom{\beta_i}{j} (1 - t)^j,
\]
we arrive at
\begin{align*}
    \operatorname{Hilb}_A(t)
        &=      \frac{1}{e_d} \sum\limits_{m=0}^\infty \frac{\varphi_m \sum_{i=1}^r t^{\beta_i}}{(1 - t)^{d - m}}
        \\&=    \frac{1}{e_d} \sum\limits_{m=0}^\infty \frac{\varphi_m}{(1-t)^{d - m}}
                    \sum_{i=1}^r \sum_{j=0}^\infty (-1)^j \binom{\beta_i}{j} (1 - t)^j
        \\&=    \frac{1}{e_d} \sum\limits_{\ell=0}^\infty \frac{1}{(1 - t)^{d - \ell}}
                    \sum_{i=1}^r \sum_{j=0}^{\beta_i} (-1)^j \binom{\beta_i}{j} \varphi_{\ell-j}
        \\&=    \frac{1}{e_d} \sum\limits_{\ell=0}^\infty \frac{\sum_{j=0}^\ell \frac{(-1)^j}{j!}
                    \varphi_{\ell-j} \sum_{k=0}^j s(j,k)p_k}{(1-t)^{d-\ell}}.
\end{align*}
The terms corresponding to $\ell=0,1$ reproduce the formulas in Equation~\eqref{eq:gammaexpressions}.
We evaluate the terms corresponding to $\ell=2$:
\begin{align*}
    \gamma_2(A)
        &=      \frac{1}{e_d} \left(r\varphi_2 - \varphi_1 p_1 + \frac{1}{2}(p_1 + p_2) \right)
        \\&=    \frac{1}{12e_d}\left( r\left(e_2 + e_1^2 - 3(d - 1)e_1 + \frac{d}{2}(3d - 5)\right)
                + 6p_2 + 6p_1 (d - 1 - e_1) \right),
\end{align*}
and the terms corresponding to $\ell=3$:
\begin{align*}
    \gamma_3(A)
        &=      \frac{1}{e_d}\left( r\varphi_3 - \varphi_2 p_1 + \frac{\varphi_1}{2} (p_2 - p_1)
                - \frac{1}{6}(2p_1 - 3p_2 + p_3) \right)
        \\&=    \frac{r}{24e_d}\left(e_2e_1 - (d - 2)(e_2 + e_1^2) + \frac{d-1}{2}(3d - 8)e_1
                - \frac{d(d - 2)(d - 3)}{2}\right)
            \\&\quad   - \frac{p_1}{12e_d} \left(e_2 + e_1^2 - 3(d - 1)e_1 + \frac{d}{2}(3d - 5)\right)
                        + \frac{1}{4e_d} (e_1 - d)(p_2 - p_1) - \frac{1}{6e_d} (2p_1 - 3p_2 + p_3).
\end{align*}


\bibliographystyle{amsplain}
\bibliography{CHHS-S1Hilb}

\end{document}